\documentclass[10pt]{amsart}
\usepackage{mathtools}
\usepackage{kpfonts}

\newtheorem{theorem}{Theorem}[section]
\newtheorem{lemma}[theorem]{Lemma}
\newtheorem{trditev}[theorem]{Proposition}

\theoremstyle{definition}

\theoremstyle{remark}
\newtheorem{remark}[theorem]{Remark}

\numberwithin{equation}{section}

\usepackage{amssymb}
\usepackage{mathdots}
\usepackage{pst-node}
\newcommand{\C}{\mathbb{C}}

\newcommand{\R}{\mathbb{R}}
\newcommand{\I}{\mathcal{I}}
\newcommand{\hra}{\hookrightarrow}
\newcommand{\CP}{\mathbb{C}\mathrm{P}}

\def\rank{\mathop{\rm rank}\nolimits}
\def\sign{\mathop{\rm sign}\nolimits}
\def\Trace{\mathop{\rm Tr}\nolimits}

\def\diag{\mathop{\rm diag}\nolimits}
\def\Rea{\mathop{\rm Re}\nolimits}
\def\Ima{\mathop{\rm Im}\nolimits}
\def\Discr{\mathop{\rm Discr}\nolimits}
\def\Res{\mathop{\rm Res}\nolimits}

\parskip=\smallskipamount

\makeatletter
\newcommand{\doublewidetilde}[1]{{%
  \mathpalette\double@widetilde{#1}%
}}
\newcommand{\double@widetilde}[2]{%
  \sbox\z@{$\m@th#1\widetilde{#2}$}%
  \ht\z@=.9\ht\z@
  \widetilde{\box\z@}%
}
\makeatother

\begin{document}
\title[]
{On Normal Forms of Complex Points of codimension 2 submanifolds}
\author{Marko Slapar and Tadej Star\v{c}i\v{c}}
\address{Faculty of Education, University of Ljubljana, Kardeljeva Plo\v{s}\v{c}ad 16, 1000 Lju\-blja\-na, Slovenia}
\address{Faculty of Mathematics and Physics, University of Ljubljana,
  Jadranska 19, 1000 Lju\-blja\-na, Slovenia}
\address{Institute of Mathematics, Physics and Mechanics, Jadranska
  19, 1000 Ljubljana, Slovenia}
\email{marko.slapar@pef.uni-lj.si}
\address{Faculty of Education, University of Ljubljana, Kardeljeva Plo\v{s}\v{c}ad 16, 1000 Lju\-blja\-na, Slovenia}
\address{Institute of Mathematics, Physics and Mechanics, Jadranska
  19, 1000 Ljubljana, Slovenia}
\email{tadej.starcic@pef.uni-lj.si}
%\curraddr{}
%\email{tadej.starcic@pef.uni-lj.si}
%    General info
\date{January 8, 2018}

%\dedicatory{}

\keywords{CR manifolds, complex points, normal forms, simultaneous reduction \\
\indent Research supported by grant P1-0291 and J1-7256 
from ARRS, Republic of Slovenia.}

\begin{abstract} In this paper we present some linear algebra behind
  quadratic parts of quadratically flat complex points of codimension two real
  submanifold in a complex manifold. Assuming some extra nondegenericity and using the result of Hong, complete normal form descriptions can be given,
 and in low dimensions, we obtain 
a  complete classification without any extra assumptions. 
\end{abstract}

\maketitle

\section{Introduction} 

Let $f\colon M^{2n}\hra X^{n+1}$ be a real $2n$-dimensional smooth manifold,
embedded into an $(n+1)$-dimensional complex manifold $(X,J)$. We assume
that $f$ is at least $\mathcal C^2$-smooth. A point
$p\in M$ is called \textit{CR regular}, if the dimension of the complex 
tangent space $T^C_pM=df(T_pM)\cap Jdf(T_pM)\subset T_{f(p)}X$ equals the algebraically expected
dimension $n-1$. If the complex dimension of $T_p^CM$ equals $n$, we
call such a point \textit{complex point}. Complex points can be seen
as the intersection of $Gf(M)$ with the subbundle $\CP (T^*X,J)$, where $Gf\colon M\to f^*Gr_{2n} TX$ is the  Gauss map $p\mapsto
  [df(T_pM)]\in Gr_{2n} T_{f(p)}X$ (see \cite{Lai}).   
By Thom transversality, for generic embeddings, the intersection
is transverse and so complex points are isolated by a simple dimension
count. We call such transverse complex points \textit{nondegenerate}
and we can assign a sign to such an intersection (note that the sign
can be assigned even in the case of nonorientable $M$).  If the sign
is positive, the complex point is called \textit{elliptic}, if it is
negative, \textit{hyperbolic}.  

Topological structure of complex points was first studied by
Lai \cite{Lai} and, specifically for surfaces, mostly by Forstneri\v c \cite{F}.  Classification of complex points up
to isotopy is treated in \cite{S1, S2, S3}. Most
research into complex points has been focused into trying to
understand  local hulls of holomorphy. 
This direction was started by Bishop \cite{Bishop} and is now
well understood in the case of surfaces through the works of Kenig and
Webster \cite{KW}, Moser and Webster \cite{MW} and Huang
\cite{Huang} among others. Global theory (filling spheres with
holomorphic discs) was studied by Bedford and Gaveau \cite{BG} and
Bedford and Klingenberg \cite{BK}, and has later resulted in many
important theorems in symplectic and contact geometry.  In higher
dimensions, assuming real analyticity, a similar
problem of finding an appropriate Levi flat hypersurface that is
bounded by the submanifold near the complex point, is treated
first in the papers of Dolbeault, Tomasini and Zaitsev \cite{DTZ1,
  DTZ2} and to a greater generality by  Huang and Yin \cite{HY2, HY3} and Fang and
Huang \cite{FH}. The problem is equivalent to understanding when the
manifold can be holomorphically flattened near the complex point. 

A closely related topic is trying to understand normal forms for
manifolds near real analytic complex points. Again, the situation is well understood
in the surface case by the work of  Moser and Webster \cite{MW} and Gong \cite{Gong1,
  Gong2}, but it seems to be much more intractable in higher
dimensions. It is already an interesting problem to completely classify
complex points up to their quadratic part. 

If $p\in M$ is a complex point for a $\mathcal{C}^2$-embedding $f\colon M^{2n}\to
X^{n+1}$, then in some local coordinates near $f(p)$, $M$ is given by
an equation of the form

\begin{equation}\nonumber
w=\overline{z}^TAz+\Rea (z^TBz)+o(|z|^2),
\end{equation}

\noindent where $A$ and $B$ are some $n\times n$ complex matrices with $B$
symmetric. By applying a biholomorphic change of coordinates near $f(p)$, while
preserving a general structure of the above equation, the pair of
matrices $(A,B)$ transforms into  $(cP^*AP,\overline c P^TBP)$, where
$c\in \C$, $|c|=1$ and $P$ is a nonsingular $n\times n$ matrix. This is more carefully
explained in the next section. Classifying complex points up
to their quadratic term means finding nice normal form representatives
for matrices $A$ and $B$ for this congruence relation, and it reduces to a linear algebra problem. The
classification is trivial in the case of $n=1$, and for $n=2$  was
done by Coffman \cite{Coff}. Note that if $A$ is positive definite, we
can find a nonsingular $P$, so that $P^*AP=I$ and then use Autonne-Takagi theorem (any
complex symmetric matrix is unitary $T$-congruent to a real diagonal matrix
with non-negative entries) to simplify $B$. For general Hermitian $A$ (even for
semi-definite), we cannot simultaneously diagonalize both $A$ and $B$,
as can quickly be seen by a simple example
$A=\left[\begin{smallmatrix} 1&0\\0&0\end{smallmatrix}\right]$,
$B=\left[\begin{smallmatrix} 0&b\\ b&0\end{smallmatrix}\right]$, where
$b\ne 0$.

The purpose of this paper is to give a
better understanding of the classification
of complex points up to their quadratic term for $n>2$ in the
quadratically flat case (when $A$ is Hermitian after a multiplication
by a nonzero complex scalar). Assuming that $\det B\ne 0$ in the pair $(A,B)$, the work of Hong
\cite{Hong89, Hong90}, Hong and Horn \cite{HongHorn} and others, 
immediately gives complete normal form descriptions in all dimensions. We summarize those results in Section
\ref{sec1} (FORM 1 and FORM 2). In both these forms the matrix $B$ is
first made into the identity matrix using Autonne-Takagi theorem, and then $A$ is simplified by
$*$-congruence using complex orthogonal matrices. Since we often want the matrix $A$ to be in its simplest form, we also
introduce FORM 3, where first $A$ is diagonalized using Sylvester's
theorem and then $B$ is
simplified while preserving $A$. At the end of this section we also point out a much simpler description
in the generic case of quadratically flat complex points (Proposition
\ref{nondeg}).  
In Section \ref{sec2}, we extend these results in dimensions $n=3$ and $n=4$ 
(see Theorem \ref{clasB}), to obtain a  complete classification without
any extra assumptions on the pair $(A,B)$. While we fail to give a nice form for complete
classification in all dimensions, a quite general result is given in Lemma \ref{prepare}.

\section{Normal forms up to quadratic terms}\label{notation}

A real $2n$-manifold $M$ embedded $\mathcal{C}^2$-smoothly in a complex
$(n+1)$-manifold $X$ can locally near an isolated complex point $p\in M$ be seen as a graph:
\begin{equation}\label{BasForm1}
w=\overline{z}^TAz+\Rea (z^TBz)+o(|z|^2), \quad (w(p),z(p))=(0,0),
\end{equation}
where $(z,w)=(z_1,z_2,\ldots,z_n,w)$ are suitable local coordinates on
$X$, and $A$, $B$ are $n\times n$ complex matrices, and in addition $B$
is symmetric ($B=B^T$). A real analytic complex point $p$ is called \textit{flat},
if local holomorphic coordinates can be chosen, so that the graph of
(\ref{BasForm1}) lies in $\C^n_z\times \R\subset \C^n_z\times\C_w$. It
is called \textit{quadratically flat}, if local holomorphic coordinates can be chosen so
that the quadratic part $\overline{z}^TAz+\Rea (z^TBz)$ of (\ref{BasForm1}) is real valued
for all $z$. It is clear that this happens precisely when $A$ is Hermitian.  
 
Let $GL_n(\mathbb{C})$ denote the set of all non-singular $n\times n$
complex matrices. Any holomorphic change of coordinates that
preserves the general form of (\ref{BasForm1}) has the same effect on
the quadratic part of (\ref{BasForm1}) as a complex-linear change of the form
\[
\begin{bmatrix}
z \\
w
\end{bmatrix}
=
\begin{bmatrix}
P & b \\
0 & c
\end{bmatrix}
\begin{bmatrix}
\widetilde{z} \\
\widetilde{w}
\end{bmatrix}, \quad P\in GL_n(\mathbb{C}),\, b\in \C^n, \,c\in \mathbb{C}\setminus \{0\}
.\]
Using this linear change of coordinates, the form (\ref{BasForm1})
transforms into

\[
\widetilde{w}=\overline{\widetilde{z}}^T\left(\frac{1}{c}P^{*}AP\right)\widetilde{z}+
\frac{1}{2}{\widetilde{z}}^T\left(\frac{1}{c}P^TBP\right)\widetilde{z}+\frac{1}{2}{\overline{\widetilde{z}}}^T
\overline{\left(\frac{1}{\overline c}P^T B P\right)}
\overline{\widetilde{z}} +\frac{1}{c} b^T \widetilde{w}+o(|\widetilde{z}|^2)
\]

\noindent or

\[
\widetilde{w}-\frac{1}{c} b^T
\widetilde{w}-\frac{1}{2}{\widetilde{z}}^T\left(\frac{\overline c-c}{|c|^2}P^TBP\right)\widetilde{z}=
\overline{\widetilde{z}}^T\left(\frac{1}{c}P^{*}AP\right)\widetilde{z}+
\Rea\left({\widetilde{z}}^T\left(\frac{1}{\overline c}P^T B P\right)
\widetilde{z}\right) +o(|\widetilde{z}|^2).
\]

\noindent If we denote \(\widehat{w}=\widetilde{w}-\frac{1}{c} b^T
\widetilde{w}-\frac{1}{2}{\widetilde{z}}^T\left(\frac{\overline
    c-c}{|c|^2}P^TBP\right)\widetilde{z},\) we get the equation
\begin{equation}\label{BasForm2}
\widehat{w}=\overline{\widetilde{z}}^T\left(\frac{1}{c}P^{*}AP\right) \widetilde{z}+\Rea \left(\widetilde{z}^T \left(\frac{1}{\overline{c}} P^TBP\right) \widetilde{z}\right)+o(|\widetilde{z}|^2).
\end{equation}

%The goal is to find a canonical normal form for a pair of $n\times n$ complex matrices $(A,B)$ up to a $*$-congruence (respectively $T$-congruence) with some non-singular matrix $P$ and multiplication by a non-zero complex number $c$ (respectively $\overline{c}$) with respect to $A$ (respectively $B$).

It is clear that by real scaling the determinant of $P$, we can assume
that $|c|=1$. Therefore the so-called {\it $\sim$-congruence} is introduced:
\begin{equation}\label{sim-cong}
(A,B)\sim (cP^{*}AP,\overline{c}P^TBP), \quad P\in GL_n(\mathbb{C}), \,c\in \mathbb{C}\setminus \{0\}, |c|=1. 
\end{equation}
It is clearly an equivalence relation and one can study its
equivalence classes by trying to find a canonical normal form for a
pair $(A,B)$. Note that $P^{*}AP$ is Hermitian if and only if $A$ is
Hermitian, so a complex point of the form (\ref{BasForm1}) is
quadratically flat precisely when $A$ is Hermitian up to a nonzero complex
scalar multiple. It thus suffices in this case to consider that $A$ is Hermitian, $c\in\{1,-1\}$ in (\ref{sim-cong}).

The case $n=1$ is very well understood and the complex points are
always quadratically flat and locally given by an equation
$w=z\overline{z}+\frac{\gamma}{2} (z^2+\overline{z}^2)+o(|z|^2)$, $0\leq \gamma
<\infty$ or $w=z^2+\overline{z}^2+o(|z|^2)$ (Bishop
\cite{Bishop}). If they are real analytic and elliptic  ($\gamma<1$),
they are also flat (see \cite{MW}). In dimension $n=2$ a simple
description of $\sim$-congruence classes can be obtained (Coffman
\cite{Coff} and Izotov \cite{Izotov}). In a later section we
generalize this result to dimensions $3$ and $4$ in the case of
quadratically flat complex points. 
%Since this is precisely the case when $A$ in (\ref{BasForm1}) is Hermitian up to multiplication by a nonzero complex number, it suffices in this case to consider that $A$ is Hermitian, $c\in\{1,-1\}$ in (\ref{sim-cong}). 
%Moreover, if in addition $B$ is non-singular the classification had been done even in higher dimensions (Hong \cite{Hong90} and Bernhardsson \cite{Bern}). 

%
%
%
%
\section{Normal forms for a pair of one Hermitian matrix and one symmetric nonsingular matrix}\label{sec1}

Throughout the paper we denote by $M_1\oplus M_2\cdots\oplus M_r$,
$r\in \mathbb{N}$, the block-diagonal matrix with blocks
$M_1,\ldots,M_r$ on the main diagonal. Next, for any $m\in \mathbb{N}$,
the matrices $I_m$ and $0_m$ will be the $m\times m$ identity-matrix
and $m\times m$ zero-matrix. (Sometimes, the index $m$ might be omitted.)

By Autonne-Takagi factorization (see e.g. \cite[Corolarry
4.4.4]{HornJohn}), a complex symmetric $n\times n$ matrix $B$ is
unitary $T$-congruent to a diagonal matrix $\Lambda \oplus 0_{n-m}$,
where $0\leq m\leq n$ and possibly $\Lambda$ is a $m\times m$ diagonal
matrix with positive diagonal entries. Further, by $*$-conjugating $A$
and by $T$-conjugating $B$ with $\sqrt{\Lambda^{-1}}\oplus I_{n-m}$, we achieve that $(A,B)\sim (\widetilde{A},I_m\oplus 0_{n-m})$ for some $\widetilde{A}$. Observe that if $B$ is nonsingular, our situation reduces to the case when $B=I$ is the identity-matrix.

Next, $*$-conjugating ($T$-conjugating) with $iI$ preserves the matrix (multiplies the matrix by $-1$). Hence, for $m_1\times m_1$ matrices $A_1,B_1$ and $m_2\times m_2$ matrices $A_2,B_2$ the conjugation with $iI_{m_1}\oplus I_{m_2}$ gives
\begin{equation}\label{iI}
(A_1\oplus A_2,B_1\oplus B_2)\sim(A_1\oplus A_2,-B_1\oplus B_2),
\end{equation}
and in particular, $\sim$-congruence for Hermitian-symmetric pairs ($c\in\{1,-1\}$ in (\ref{sim-cong})) can be redefined as: 
\begin{equation}\label{sim-cong2}
(A,B)\sim (\pm P^{*}AP,P^TBP), \qquad P\in GL_n(\mathbb{C}).
\end{equation}
For Hermitian matrices $A$, $\widetilde{A}$, it then follows that
$(A,I)\sim(\widetilde{A},I)$ is equivalent to the existence of a
complex orthogonal matrix $Q$ and a constant $\varepsilon\in\{1,-1\}$,
such that $\widetilde{A}=\varepsilon Q^{*}AQ=\varepsilon
\overline{Q}^{-1}AQ$, thus the concept of $\sim$-congruence is closely
related to the notion of {\it consimilarity}, i.e. $A$ and
$\widetilde{A}$ (not necessarily Hermitian) are consimilar if and only
if there exists a non-singular (not necessarily orthogonal) matrix $Q$
such that $\widetilde{A}=\overline{Q}^{-1}AQ$. Also, note that the
problem of consimilarity is further connected to the problem of
similarity of certain corresponding matrices. We refer to
\cite{HornJohn} or to section \ref{appendix}. Appendix for the
basic properties on this topics. Based on the work of several authors
(see the discussion below) it was eventually observed by Hong
\cite{Hong90} and Bernhardsson \cite{Bern} that the following two normal forms can be obtained for a pair $(A,I)$ with $A$ Hermitian.

\begin{enumerate}
\item[{\bf FORM 1.}] $(\mathcal{H}^{\epsilon}(A),I)$,
\begin{equation}\label{NF1}
\mathcal{H}^{\epsilon}(A)=\left(\bigoplus_j \epsilon_j H_{\alpha_j}(\lambda_j)\right)\oplus\left(\bigoplus_k K_{\beta_k}(\mu_k)\right)\oplus\left( \bigoplus_l L_{\gamma_l}(\xi_l)\right),
\end{equation}
where $\epsilon=\{\epsilon_j\}$ with $\epsilon_j\in\{1,-1\}$,
$\lambda_j\in \mathbb{R}_{\ge 0} $, $\mu_k\in \mathbb{R}_{>0}$, $\xi_l\in
\mathbb{C}\setminus \mathbb{R}$ are such that $\lambda_j^2$,
$-\mu_k^2$ and $\xi_l^2$ are real non-negative, real negative and
non-real eigenvalues of $A\overline{A}$, respectively; 
\begin{equation}\label{Hmz}
H_m(z)=
\frac{1}{2}\left(
\begin{bmatrix}
0  &                &    1  & 2z\\
 &   \iddots           &   \iddots     & 1 \\
1   &  \iddots  & \iddots &  \\
2z   & 1          &     & 0 \\
\end{bmatrix}
+
i
\begin{bmatrix}
0  & 1 &             & 0 \\
-1 & \ddots &   \ddots        &   \\
   & \ddots  &  \ddots & 1 \\
0   &          & -1    & 0 \\
\end{bmatrix}
\right)\in \mathbb{C}^{m\times m},
\end{equation}
\begin{equation}\label{KLm}
K_{m}(z)=
\begin{bmatrix}
0       & -iH_m(z)    \\
iH_m(z) &  0    \\
\end{bmatrix}, \qquad
L_{m}(z)=
\begin{bmatrix}
0       & H_m(z)    \\
H^{*}_m(z) &  0    \\
\end{bmatrix}.
\end{equation}
\item[{\bf FORM 2.}] $(\mathcal{J}_E^{\epsilon}(A),E(A))$,
\begin{align}\label{NF2}
&\mathcal{J}_E^{\epsilon}(A)=\left(\bigoplus_j \epsilon_j E_{\alpha_j}J_{\alpha_j}(\lambda_j,1)\right)\oplus\left( \bigoplus_k E_{2\beta_k}J_{\beta_k}(\Lambda_k,I_2)\right),\\
& E(A)=\left(\bigoplus_j E_{\alpha_j}\right)\oplus\left( \bigoplus_k E_{2\beta_k}\right),\nonumber
\end{align}
where $\epsilon=\{\epsilon_j\}$ with $\epsilon_j\in\{1,-1\}$, and
$\lambda_j \in \mathbb{R}_{\ge 0}$, $\Lambda_k=\begin{bsmallmatrix}  
	                                            a_k & -b_k\\
	                                            b_k & a_k
                                       \end{bsmallmatrix}$ with $a_k
                                     \in \mathbb{R}_{\ge 0}$, $b_k\in \mathbb{R}\setminus\{0\}$ are such that $\lambda_j$, $a_k+i b_k$ are the eigenvalues of the corresponding double-sized matrix 
$
\begin{bsmallmatrix}
0 & \overline{A} \\
A & 0
\end{bsmallmatrix}
$; we denoted by $E_m$ the backward $m$-identity-matrix (with ones on the anti-diagonal) and 
\[
 J_m(C,D)=\begin{bmatrix}
                                                      C    &  D       & \;     & \;    \\
						      \;     & C     & \ddots & \;    \\     
						      \;     & \;      & \ddots &  D     \\
                                                      \;     & \;      & \;     & C   
                                   \end{bmatrix}\in \mathbb{C}^{rm\times rm},\,\,
  \qquad C,D\in \mathbb{C}^{r\times r}.                                 
\]
\end{enumerate}

For the sake of completeness, we recall the algorithmic procedure, how to get the normal FORM 1 (in five steps):
\begin{itemize}
\item The corresponding matrix $A\overline{A}$ is similar to its Jordan canonical form 
\[\hspace{12mm}
\mathcal{J}(A\overline{A})=\left(\bigoplus_j J_{\alpha_j}(\lambda_j^2,1)\right)\oplus \left(\bigoplus_k J_{\beta_k}(-\mu_k^2,1) \right)\oplus \left(\bigoplus_l (J_{\gamma_l}(\xi_l,1)\oplus J_{\gamma_l}(\overline{\xi}_l,1)) \right),
\]
where $\lambda_j\in \mathbb{R}_{\ge 0} $, $\mu_k\in \mathbb{R}_{>0}$, $\xi_l\in \mathbb{C}\setminus \mathbb{R}$.
\item (Hong \cite{Hong90}, Hong and Horn \cite[Theorem 3.1]{HongHorn}, Hua \cite{Hua}) $A$ is consimilar to a quasi-Jordan form
%i.e. there exists a non-singular matrix $T$ such that $TA\overline{T}^{-1}=
\[
\hspace{12mm}\mathcal{J}_{q}(A)=\left(\bigoplus_j J_{\alpha_j}(\lambda_j,1)\right)\oplus\left(
\bigoplus_k 
\begin{bsmallmatrix}
0  & J_{\beta_k}(\mu_k,1) \\
-J_{\beta_k}(\mu_k,1) &  0    
\end{bsmallmatrix}\right)\oplus\left( \bigoplus_l 
\begin{bsmallmatrix}
0   & J_{\gamma_l}(\xi_l,1)    \\
\overline{J_{\gamma_l}(\xi_l,1)} &  0  
\end{bsmallmatrix}\right).
\]
Note that the blocks corresponding to the eigenvalue $0$ are uniquely determined by the so-called alternating-product rank condition \cite[Theorem 4.1]{HornJohn} (see section \ref{appendix}. Appendix).
\item (Bevis, Hall and Hartwig \cite{BHH}) $\mathcal{J}_{q}(A)$ is consimilar to $(\mathcal{J}_{q}(A))^{*}$ with a symmetric matrix $S=\left(\bigoplus_j E_{\alpha_j}\right)\oplus\left( \bigoplus_k -i(E_{\beta_k}\oplus E_{\beta_j})\right)\oplus\left( \bigoplus_l (E_{\gamma_l}\oplus E_{\gamma_l})\right)$.
\item (Hong \cite[Lemma 2.1]{Hong90}) 
$A$ is consimilar to a Hermitian matrix of the form (\ref{NF1}) with $\epsilon_j=1$ for all j:
\[
\mathcal{H}^{1}(A)=P^{-1}\mathcal{J}_{q}(A)\overline{P},
\]
where $S=P^TP$ with $P=\left(\bigoplus_j P_{\alpha_j}\right)\oplus\left( \bigoplus_k e^{\frac{i\pi}{4}}(P_{\beta_k}\oplus P_{\beta_j})\right)\oplus\left(\bigoplus_l (P_{\gamma_l}\oplus P_{\gamma_l})\right)$, and $P_m=\frac{e^{{\scriptscriptstyle -\frac{i\pi}{4}}}}{\sqrt{2}}(I_m+iE_m)$. (Note that $\mathcal{H}^{1}(A)$ is a Hermitian canonical form for $A$ and it is unique up to a permutation of the diagonal blocks.)
\item (Hong \cite[Theorem 2.7]{Hong89}) $A$ is consimilar with a complex orthogonal matrix to a Hermitian matrix $\mathcal{H}^{\epsilon}(A)$ of the form (\ref{NF1}) with $\epsilon=\{\epsilon_j\}$, $\epsilon_j\in\{1,-1\}$. (It implies $(A,I)\sim (\mathcal{H}^{\epsilon}(A),I)$.) 
\end{itemize}

%Recall that $(A,I)\sim(\mathcal{H}^{\epsilon}(A),I)$ (i.e. \cite[Theorem 2.7]{Hong89}) is equivalent to consimilarity of $A$ and $\mathcal{H}^{\epsilon}(A)$ with an orthogonal matrix. 
Trivially, $\mathcal{H}^1(A)$ is consimilar to $\mathcal{H}^{\epsilon}(A)$ for any $\epsilon$, hence $\mathcal{H}^{\epsilon}(A)$ is unique up to a permutation of its diagonal blocks as well. However, $\epsilon$ might not be uniquely determined; note that $Q^*H_3(0)Q=-H_3(0)$, where $Q=-1\oplus 1\oplus -1$ and $H_3(0)$ is of the form (\ref{Hmz}) for $m=3$, $z=0$. On the other hand, the matrices $-K_m(z)$ and $-L_m(z)$ are orthogonally $*$-congruent to $K_m(z)$ and $L_m(z)$ respectively (see e.g. \cite[Lemma 2.6]{Hong89}), hence using (\ref{sim-cong2}) we get:
\begin{equation}\label{pme}
(\mathcal{H}^{-\epsilon}(A),I)\sim (\pm \mathcal{H}^{\epsilon}(A),I).
\end{equation}
%$\varepsilon \mathcal{H}^{\epsilon}(A)$ is orthogonally $*$-congruent to $\mathcal{H}^{\varepsilon \epsilon}(A)$.

A different and more direct approach was used by Bernhardsson \cite[Theorem 6]{Bern} to obtain the normal form FORM 2. It is based on the ideas of a similar result for a pair of two symmetric matrices (Trott \cite{Trott} and Uhlig \cite{Uhlig}). Observe now that one easily transforms FORM 2 to FORM 1. Clearly, it suffices to take one Jordan block $J$ in (\ref{NF2}). By Autonne-Takagi-factorization \cite[Corolarry 4.4.4]{HornJohn} of the backward-identity matrix $E$, there exists a unitary matrix $U$ (i.e. $U^{*}U=I$) such that $U^TEU=I$. Since $(EJ,E)\sim (U^{*}EJU,I)$ and $E=(U^*)^TU^*=U \overline{U}^{-1}$ we get 
\[
\big(U^{*}EJU\big)\bigl(\overline{U^{*}EJU}\bigr)=U^{*}U \overline{U}^{-1}JU U^{T}(U^{*})^TU^{*}J\overline{U}=\overline{U}^{-1}J^2\overline{U}.
\]
Moreover, a similar computation also yields that the matrices $U^{*}EJU$ and $J$ satisfy the alternating-product rank condition, and it then follows that $\mathcal{H}^{\epsilon}(U^{*}EJU)$ of the form (\ref{NF1}) consists of only one block and it corresponds to the same eigenvalue as $J$ (see Proposition \ref{ConSP}). To sum up, given a pair $(A,I)$ with $A$ Hermitian, we see that the dimensions of the blocks and the corresponding eigenvalues of the normal forms (\ref{NF1}) and (\ref{NF2}) for $A$ agree.

We now try to find a normal form for a Hermitian-symmetric pair with a nice Hermitian matrix. By Sylvester's inertia theorem there exists a non-singular matrix $Q$ such that
\begin{equation}\label{inertia}
\mathcal{I}(A)=Q^{*}AQ=-I_p\oplus I_r\oplus 0_{n-p-r}, \qquad 
\end{equation}
called the {\it inertia matrix} of a $n\times n$ matrix $A$.
%, is a diagonal matrix with only minus-ones, ones and zeros 
%$1$, $-1$, $0$ as the diagonal entries. 
Moreover,  
the columns $q_j,\ldots,q_n$ of $Q$ are orthogonal eigenvectors with the corresponding (real) eigenvalues $\lambda_1,\ldots,\lambda_n$ of $A$ and such that $|q_j|=\frac{1}{\sqrt{|\lambda_j|}}$ for $\lambda_j\neq 0$, and it can be assumed $|q_j|=1$ if $\lambda_j=0$. By setting 
\begin{equation}\label{NBA}
\mathcal{N}_{B}(A)=Q^{T}BQ,
\end{equation}
 we have $(A,B)\sim (\mathcal{I}(A),\mathcal{N}_{B}(A))$.
Note that $\mathcal{N}_B(A)$ is determined up to the choice of the ordered orthogonal bases of the eigenspaces of $A$ of dimension greater than one. Using the normal form FORM 1 and (\ref{inertia}),(\ref{NBA}) one can get the third normal form for $(A,I)\sim (\mathcal{H}^{\epsilon}(A),I)$:
\begin{enumerate}
\item[{\bf FORM 3.}] $(\mathcal{I}(A),\mathcal{N}_I(A))$,
\begin{align}\label{NF3}
&\mathcal{I}(A)=\left(\bigoplus_j \epsilon_j \mathcal{I}\big(H_{\alpha_j}(\lambda_j)\big)\right)\oplus\left( \bigoplus_k \mathcal{I}\big(K_{\beta_k}(\mu_k)\big)\right)\oplus\left( \bigoplus_l \mathcal{I}\big(L_{\gamma_l}(\xi_l)\big)\right),\\ \nonumber
&\mathcal{N}_{I}(A)=\left(\bigoplus_j \mathcal{N}_{I_{\alpha_j}}\big(H_{\alpha_j}(\lambda_j)\big) \right)\oplus \left(\bigoplus_k \mathcal{N}_{I_{2\beta_k}}\big(K_{\beta_k}(\mu_k)\big) \right)\oplus \left( \bigoplus_l \mathcal{N}_{I_{2\gamma_l}}\big(L_{\gamma_l}(\xi_l)\big) \right),
\end{align}
where $\epsilon=\{\epsilon_j\}$ with $\epsilon_j\in\{1,-1\}$,
$\lambda_j\in \mathbb{R}_{\ge 0} $, $\mu_k\in \mathbb{R}_{>0}$,
$\xi_l\in \mathbb{C}\setminus \mathbb{R}$ are such that $\lambda_j^2$,
$-\mu_k^2$ and $\xi_l^2$ are real and non-real eigenvalues of
$A\overline{A}$, respectively; and $H_{\alpha_j}(\lambda)$ and $K_{\beta_k}(\mu_k)$, $L_{\gamma_l}(\xi_l)$ are of the form (\ref{Hmz}) or (\ref{KLm}).
\end{enumerate}

For few small $m$'s we can diagonalize Hermitian matrices $H_m(x)$, $K_m(y)$, $L_m(z)$, $x\in \mathbb{R}_{\geq 0}$, $y\in \mathbb{R}_{>0}$, $z\in \mathbb{C}\setminus\mathbb{R}$. If the matrices are of dimension $3$ or $4$ the computations get very tedious but are still manageable, unfortunately the form of the eigenvectors is not so simple. The problem to diagonalize matrices of bigger sizes lies clearly in finding their eigenvalues.

We start with $H_m(x)$, $x\in \mathbb{R}_{\geq 0}$: 
\begin{itemize}
\item $H_1(x)=x$.
\item $H_2(x)=\frac{1}{2}
\begin{bmatrix}
1 & 2x+i\\
2x-i & 1
\end{bmatrix}$,\quad \\ 
$\Delta_{H_2(x)}(\lambda)=\lambda^2-\lambda-x^2,\quad \lambda_{1,2}(x)=\frac{1}{2}(1\pm \sqrt{1+4x^2})$,\\
%$\Lambda(x)=\frac{1}{2}\diag (1-w,1+w)$, $\lambda(0)=\diag (0,1)$,\\ 
$Q(x)=
\begin{bmatrix}
-\frac{w}{|w|\sqrt{|w|-1}} & \frac{w}{|w|\sqrt{1+|w|}}\\
\frac{1}{\sqrt{|w|-1}} & \frac{1}{\sqrt{1+|w|}}
\end{bmatrix}
$ \textrm{ for } $x> 0$, $w=2x+i$, \quad$Q(0)=\frac{1}{\sqrt{2}}
\begin{bmatrix}
i & -i\\
1 & 1
\end{bmatrix}$,\\ 
$\mathcal{N}_{I_2}(H_2(x))=\left\{
\begin{array}{ll}
\tfrac{1}{2x-i}
\begin{bmatrix}
\frac{4x}{\sqrt{1+4x^2}-1} & -\frac{i}{x}\\
-\frac{i}{x} & \frac{4x}{\sqrt{1+4x^2}+1}
\end{bmatrix}, & x> 0\\
\begin{bmatrix}
0 & 1\\
1 & 0
\end{bmatrix}, & x=0
\end{array}
\right.$.
\item $H_3(x)=\frac{1}{2}
\begin{bmatrix}
0 & 1+i & 2x\\
1-i & 2x & 1+i\\
2x & 1-i & 0
\end{bmatrix}$,\\ $\Delta_{H_3(x)}(\lambda)=-\lambda^3+x\lambda^2+(x^2+1)\lambda-x^3$\,
($\Delta_{H_3(0)}(\lambda)=-(\lambda+1)(\lambda-1)\lambda$),\\ $Q(0)=\frac{1}{2}
\begin{bmatrix}
i & i & -i\sqrt{2}\\
-1-i & 1+i & 0\\
1 & 1 & \sqrt{2}
\end{bmatrix}$, $\mathcal{N}_{I_3}(H_3(0))=\frac{1}{2}
\begin{bmatrix}
i & -i & \sqrt{2}\\
-i & i & \sqrt{2}\\
\sqrt{2} & \sqrt{2} & 0
\end{bmatrix}.$\\
For $x\in \mathbb{R}_{>0}$ the matrix $H_3(x)$ has one negative and two positive eigenvalues and no multiple eigenvalues (since the discriminant $\Discr(H_3(x))=32x^4+13x^2+4$ never vanishes.) It is straightforward but somewhat tedious to compute that
$\mathcal{N}_{I_3}(H_3(x))=-i[r(\lambda_j,\lambda_k)s(\lambda_j)s(\lambda_k)]_{j,k=1}^3$, where $\lambda_{1}\leq \lambda_2\leq \lambda_3$ are the eigenvalues of $H_3(x)$ and
\begin{align*}
&r(\lambda,\mu)=\lambda^2\mu^2-x^2(\lambda^2+\mu^2)+x(\lambda+\mu)+x^4,\,\\
&s(\lambda)=\tfrac{\sqrt{\lambda^2+x^2}}{(x-i\lambda)\sqrt{|\lambda|(\lambda^4+\lambda^2(1-2x^2)+x^2+x^4)}}.
\end{align*}
\item $H_4(x)=\frac{1}{2}
\begin{bmatrix}
0 & i  & 1 & 2x\\
-i & 1 & 2x+i & 1\\
1 & 2x-i & 1 & i\\
2x & 1 & -i & 0
\end{bmatrix}$,\quad \\
$\Delta_{H_4(x)}(\lambda)=\lambda^4-\lambda^3-(2x^2+1)\lambda^2+(x^2+1)\lambda+x^4$,\\
Since $\Discr(H_3(x))=400x^8+204x^6+93x^4+32x^2$, the matrix $H_4(x)$ has multiple eigenvalues precisely for $x=0$.  
\end{itemize}

We proceed with $K_m(y)$, $y\in\mathbb{R}_{>0}$:
\begin{itemize}
\item $K_1(y)=
\begin{bmatrix}
0 & -iy\\
iy & 0
\end{bmatrix}$,\quad $\Delta_{K_1(y)}(\lambda)=\lambda^2-y^2$,\quad $\lambda_{1,2}(y)=\pm y$,\\
%$\Lambda(y)=\diag (-y,y)$, 
$Q(y)=\frac{1}{\sqrt{2y}}
\begin{bmatrix}
i & -i\\
1 & 1
\end{bmatrix}$,\quad
$\mathcal{N}_{I_2}(K_1(y))=\tfrac{1}{y}
\begin{bmatrix}
0 & 1\\
1 & 0
\end{bmatrix}$.
\item $K_2(y)=\frac{1}{2}
\begin{bmatrix}
0 & 0  & -i & 1-2iy\\
0 & 0 & -1-2iy & -i\\
i & -1+2iy & 0 & 0\\
1+2iy & i & 0 & 0
\end{bmatrix}$,\quad\\
$\Delta_{K_2(y)}(\lambda)=\lambda^4-(2y^2+1)\lambda^2+y^4$, \quad $\lambda_{1,2,3,4}=\frac{1}{2}(\pm 1\pm \sqrt{1+4y^2})$,\\
%$\Lambda(y)=\frac{1}{2}\diag (-1-w,1-w,-1+w,1+w)$, $w=2y+i$,\\
$Q(y)=\frac{1}{2}
\begin{bmatrix}
\frac{iw}{|w|\sqrt{|w|+1}} & \frac{iw}{|w|\sqrt{1-|w|}} & -\frac{iw}{|w|\sqrt{1-|w|}} & -\frac{iw}{|w|\sqrt{|w|+1}}\\
\frac{i}{|w|\sqrt{|w|+1}} & -\frac{i}{\sqrt{1-|w|}} & \frac{i}{\sqrt{1-|w|}} & -\frac{i}{\sqrt{|w|+1}}\\
\frac{w}{|w|\sqrt{|w|+1}} & -\frac{w}{|w|\sqrt{1-|w|}} & -\frac{w}{|w|\sqrt{1-|w|}} & \frac{w}{|w|\sqrt{|w|+1}}\\
\frac{1}{|w|\sqrt{|w|+1}} & \frac{1}{\sqrt{1-|w|}} & \frac{1}{\sqrt{1-|w|}} & \frac{1}{\sqrt{|w|+1}}
\end{bmatrix}$, $w=2y+i$,\\
$\mathcal{N}_{I_4}(K_2(y))=\frac{1}{y(2y-i)}
\begin{bmatrix}
0 & -i  & 0 & \frac{4y}{1+\sqrt{1+4y^2}}\\
-i & 0 & \frac{4y}{-1+\sqrt{1+4y^2}} & 0\\
0 & \frac{4y}{-1+\sqrt{1+4y^2}} & 0 & -i\\
\frac{4y}{1+\sqrt{1+4y^2}} & 0 & -i & 0
\end{bmatrix}$.
\end{itemize}

Finally, we inspect $L_m(z)$, $z\in\mathbb{C}\setminus\mathbb{R}$:
\begin{itemize}
\item $L_1(z)=
\begin{bmatrix}
0 & z\\
\overline{z} & 0
\end{bmatrix}$,\quad $\Delta_{L_1(z)}(\lambda)=\lambda^2-|z|^2$, \quad $\lambda_{1,2}(z)=\pm |z|$,\\
%$\Lambda(z)=\diag (-|z|,|z|)$, 
$Q(z)=\frac{1}{\sqrt{2|z|}}
\begin{bmatrix}
-\frac{z}{|z|} & \frac{z}{|z|}\\
1 & 1
\end{bmatrix}$,\quad
$\mathcal{N}_{I_2}(L_1(z))=\tfrac{1}{2\overline{z}|z|}
\begin{bmatrix}
z+\overline{z} & \overline{z}-z\\
\overline{z}-z & z+\overline{z}
\end{bmatrix}$.
\item $L_2(z)=\frac{1}{2}
\begin{bmatrix}
0 & 0  & 1 & 2z+i\\
0 & 0 & 2z-i & 1\\
1 & 2\overline{z}+i & 0 & 0\\
2\overline{z}-i & 1 & 0 & 0
\end{bmatrix}$,\quad \\
$\Delta_{L_2(z)}(\lambda)=\lambda^4-(2|z|^2+1)\lambda^2+|z|^4$,\quad $\lambda_{1,2,3,4}=\frac{1}{2}(\pm 1\pm \sqrt{1+4|z|^2})$\\
%$D_2(z)=\frac{1}{2}\diag (-1-w,1-w,-1+w,1+w)$, $w=\sqrt{1+4|z|^2}$.\\
Using Lemma \ref{lemaL4} in section \ref{appendix}. Appendix we can now compute $\mathcal{N}_{I_4}(L_2(z))$.
\end{itemize}

Observe that when $A\overline{A}$ is diagonalizable,
$\mathcal{H}^{\epsilon}(A)$ has only blocks of the form $\pm H_1(x)$,
$x\in\mathbb{R}_{\ge 0}$ and $L_1(\xi)$, $\xi\in
\mathbb{C}\setminus{\mathbb{R}}$ (since $L_1(-iy)=K_1(y)$ for $y\in
\mathbb{R}_{>0}$). Furthermore, if the discriminant
$\Discr(\Delta_{A\overline{A}})$ (and hence the resultant $\Res
(\Delta_{A\overline{A}},\Delta_{A\overline{A}}')$ does not vanish,
which is true for matrices on an open dense subset of
$GL_n(\mathbb{C})$ (on a complement of a real analytic subset of
codimension $1$), we immediately get the following simple structures of FORM 1, FORM 2 and FORM 3:
%with $\Discr(\Delta_{A\overline{A}})=(-1)^{\frac{n(n+1)}{2}}\Res (\Delta_{A\overline{A}},\Delta_{A\overline{A}}')\neq 0$ on a dense open subset of $GL_n(\mathbb{C})$,

\begin{trditev}\label{nondeg}
If $\Discr( \Delta_{{\scriptscriptstyle A\overline{A}}})=(-1)^{{\scriptscriptstyle \frac{n(n+1)}{2}}}\Res (\Delta_{{\scriptscriptstyle A\overline{A}}},\Delta_{{\scriptscriptstyle A\overline{A}}}')\neq 0$ and $A$ is a Hermitian $n \times n$ matrix, then the normal forms for a pair $(A,I)$ are $(\mathcal{H}^{\epsilon}(A),I)$, $(\mathcal{J}_E^{\epsilon}(A),E(A))$ and $(\mathcal{I}(A),\mathcal{N}_I(A))$:
\begin{align*}
\mathcal{H}^{\epsilon}(A) &= \left(\bigoplus_{j}\epsilon_j x_{j}\right)\oplus\left(  \bigoplus_l 
\begin{bmatrix}
0 & \xi_l \\
\overline{\xi}_l & 0
\end{bmatrix}\right),\quad  \\
\mathcal{J}_E^{\epsilon}(A) & =\left(\bigoplus_j \epsilon_j x_j\right)\oplus\left( \bigoplus_l \frac{1}{2}
\begin{bmatrix}
i(\xi_l-\overline{\xi}_l) & \overline{\xi}_l+\xi_l\\
\overline{\xi}_l+\xi_l & i(\overline{\xi}_l-\xi_l)
\end{bmatrix}\right),\quad
E(A)  =I_{j}\oplus\left( \bigoplus_l \begin{bmatrix}   
	                                            0 & 1\\
	                                            1 & 0
                                       \end{bmatrix}\right),\\
\mathcal{I}(A) &=\left(\bigoplus_j \widetilde{\epsilon}_j\right)\oplus\left( \bigoplus_l\diag (-1,1)\right), \quad 
\widetilde{\epsilon}_j=\left\{
\begin{array}{ll}
\epsilon_j, & x_j\neq 0 \\
0,         & x_j=0
\end{array}\right.,\\
\mathcal{N}_I(A) &=\left(\bigoplus_{j}x_{j}^{-1}\right)\oplus\left( \bigoplus_l \tfrac{1}{2\overline{\xi}_l|\xi_l|}
\begin{bmatrix}
\xi_l+\overline{\xi}_l & \overline{\xi}_l-\xi_l\\
\overline{\xi}_l-\xi_l & \xi_l+\overline{\xi}_l
\end{bmatrix}\right), \quad 
\widetilde{x}_j=\left\{
\begin{array}{ll}
x_j^{-1}, & x_j\neq 0 \\
1,         & x_j=0
\end{array}\right.,
\end{align*}
where $\epsilon_j\in\{1,-1\}$ and $x_{j}\in \mathbb{R}_{\geq 0}$, $\xi_l\in \mathbb{C}\setminus{\mathbb{R}}$ are all distinct constants.
\end{trditev}

\section{Normal forms for Hermitian-symmetric pair of matrices of dimensions $3$ and $4$}\label{sec2}

In dimensions $2$, $3$ and $4$, we can give a complete list of normal forms for Hermitian-symmetric pairs $(A,B)$ with possibly $B$ singular. In a more general setting, we first prove the following preparatory lemma. We note that the result of Hong \cite[Theorem 2.7]{Hong89} is a key ingredient in its proof.

\begin{lemma}\label{prepare}
Let $A$ be an $n\times n$ Hermitian matrix and let $B=I_m\oplus 0_{n-m}$ with $0\leq m\leq n$. Then there exists an $n\times n$ matrix $\widetilde{A}$ satisfying $(A,B)\sim (\widetilde{A},B)$, and such that it is of the form 
\begin{equation}\label{NFp}
\widetilde{A}=
\begin{bmatrix} \mathcal{H}_m^{\epsilon} & 0& Y\\
                             0 & \I & 0\\
                             Y^* &  0 & 0   
\end{bmatrix}.
\end{equation}
Here $\mathcal{I}$ is a non-singular diagonal $(n-m-k)\times (n-m-k)$ matrix with only minus-ones and ones on the diagonal, $\Trace(\mathcal{I})\geq 0$, $n-m\geq k \geq 0$, $\mathcal{H}_m^{\epsilon}$ is a $m\times m$ matrix of the form (\ref{NF1}), and $Y$ is a $m\times k$ matrix, determined up-to right-multiplication with a non-singular $k\times k$ matrix. 
Furthermore, if $\rank Y=k\leq m$, then
\[\widetilde{A}=\begin{bmatrix} 0& 0 & 0 & I_k\\
                                 0 & \mathcal{H}_{m-k}^{\epsilon} & 0 & L\\
                                 0 & 0 & \I & 0 \\
                                 I_k & L^* &0 &0 
\end{bmatrix},\] 
where $\mathcal{H}_{m-k}^{\epsilon}$ is an $(m-k)\times (m-k)$ matrix of the form (\ref{NF1}), and $L$ is an arbitrary $(m-k)\times k$ matrix. If in addition, $I+L^TL$ and $I+LL^T$ are both nonsingular, we can have $L=0$.
\end{lemma}

\begin{proof}
Recall that for $n\times n$ Hermitian matrices $A,\widetilde{A}$ and $B=I_m\oplus 0_{n-m}$ we have $(A,B)\sim (\widetilde{A},B)$ if and only if there exists a matrix $T\in GL_n(\mathbb{C})$ and a constant $\varepsilon\in\{1,-1\}$ such that $\widetilde{A}=\varepsilon T^{*}AT$, $\widetilde{B}=T^TBT$ (see (\ref{sim-cong2})). 
The later equality holds precisely for matrices of the form  
\begin{equation}\label{stabB}
T=
\begin{bmatrix}
P & 0 \\
R & S
\end{bmatrix}
,
\end{equation}
 where $0$ denotes an $m\times (n-m)$ zero-matrix, $P$ is an $m\times m$ complex orthogonal matrix ($P^TP=I$), $R$ is an arbitrary $(n-m)\times m$ matrix and $S\in GL_{n-m}(\mathbb{C})$. Indeed, 
\[
\begin{bmatrix}
I_m & 0 \\
0 & 0_{n-m}
\end{bmatrix}
=
\begin{bmatrix}
P & Q \\
R & S
\end{bmatrix}^T
\begin{bmatrix}
I_m & 0 \\
0 & 0_{n-m}
\end{bmatrix}
\begin{bmatrix}
P & Q \\
R & S
\end{bmatrix}
=
\begin{bmatrix}
P^T P & P^T Q \\
Q^{T}P & Q^T Q
\end{bmatrix}.
\]

We can write a Hermitian $n\times n$ matrix $A$ in the form
\begin{equation}\label{formA}
A=
\begin{bmatrix} H &X \\
X^* & E
\end{bmatrix},
\end{equation}
where $H$ is an $m\times m$ Hermitian, $E$ is an $(n-m)\times (n-m)$
Hermitian and $X$ is an arbitrary $m\times (n-m)$ complex matrix.
Then $*$-conjugation with the matrices of the form (\ref{stabB}), yields:
\begin{align}\label{conj}
\begin{bmatrix}
\widetilde{H} & \widetilde{X} \\ \nonumber
\widetilde{X}^{*} & \widetilde{E}
\end{bmatrix}&=
\varepsilon
\begin{bmatrix}
 P & 0 \\
R & S
\end{bmatrix}^*
\begin{bmatrix}
H & X \\
X^{*} & E
\end{bmatrix}
\begin{bmatrix}
P & 0 \\
R & S
\end{bmatrix}\\
&
=
\varepsilon
\begin{bmatrix}
P^{*}H P+R^{*}ER+R^{*}X^{*}P+(R^{*}X^{*}P)^{*} & (P^{*}X+R^{*}E)S \\
S^{*}(P^{*}X+R^{*}E)^{*} & S^{*}E S
\end{bmatrix}.
\end{align}
%Since $S$ is non-singular, we observe that the (non)-singularity of the right bottom block is preserved.
Since $E$ is Hermitian, by Sylvester's theorem we can choose an appropriate matrix $S$ (with
$P=I_m$, $R=0$) in (\ref{conj}) and a suitable $\epsilon\in\{-1,1\}$, such that $\widetilde{E}$ is of the form 
\[\widetilde{E}=\begin{bmatrix} \mathcal{I} &0 \\
0 & 0
\end{bmatrix},\] 
where $\mathcal{I}$ is $(n-m-k)\times (n-m-k)$ diagonal matrix with only $\pm 1$
on the diagonal and in addition $\Trace (\mathcal{I})\geq 0$. 
%We will from now on assume this form for $E$ in (\ref{formA}).

Taking further $\epsilon=1$, $S=I_{m-n}$, $P=I_m$ and $R=\begin{bsmallmatrix} R_1 \\
  0 \end{bsmallmatrix}$, with $R_1$ being an $(n-m-k)\times m$
matrix, transforms the
matrix $\widetilde{X}$ to  
\(
\widetilde{X}+\begin{bsmallmatrix} R_1^* & 0\end{bsmallmatrix}\begin{bsmallmatrix} \I &
    0\\0&0\end{bsmallmatrix}, 
\)
keeps $\widetilde{E}$ intact, but also sends $\widetilde{H}$ into some new Hermitian
matrix. By choosing appropriate $R_1$, we see that $\widetilde{X}$ transforms to
$\begin{bmatrix} 0 \quad Y \end{bmatrix}$,
where $Y$ is some $m\times k$ matrix. 
We may thus assume that $\widetilde{A}$ can be of the form 
\begin{equation}\label{formA2}
\widetilde{A}=\begin{bmatrix} \widetilde{H} & 0& Y\\
                             0 & \I & 0\\
                             Y^* &  0 & 0   
\end{bmatrix}.
\end{equation}
Observe that by $*$-conjugating with $P\oplus I_{n-m-k}\oplus S_k$, where $P$ is a suitably chosen $m\times m$ complex orthogonal matrix and $S_k$ is any non-singular $k\times k$ matrix, we can achieve by using \cite[Theorem 2.7]{Hong89} that $\widetilde{H}$ is of the form (\ref{NF1}); also $Y$ is determined up to right-multiplying with a $k\times k$ non-singular matrix.

Next, suppose that $Y$ has maximal rank $k$ and $m\ge k$. For $T$ in (\ref{stabB}), we take $R=0$, and $S=\begin{bsmallmatrix} I_{n-m-k} & 0  \\ 0 &
S_4 \end{bsmallmatrix}$, where $S_4\in GL_k(\mathbb{C})$ is to be chosen later. A $*$-conjugation of $\widetilde{A}$ in (\ref{formA2}) by such
$T$ preserves $\mathcal{I}$ and the zeros of $\widetilde{A}$, transforms $\widetilde{H}$
into a new Hermitian matrix, which we will by an abuse of notation
still call $\widetilde{H}$, but transforms $Y^*$ into $S_4^*Y^*P$.
We can choose a complex orthogonal matrix $P$ to first permute the columns of $Y^*$, so
that the first $k$ columns are independent, and than choose $S_4$ so
that those first columns just become the identity matrix. We can thus
assume that $Y^*$ in (\ref{formA2}) is of the form
\(
Y^*=\begin{bsmallmatrix} I_k \quad L^*\end{bsmallmatrix},
\)
where  $L$ is some $(m-k)\times k$ matrix.

We will now make an assumption that both $I+L^TL$ and $I+LL^T$ are
nondegenerate. Take $T$ of the form (\ref{stabB}) with $P=\begin{bsmallmatrix} P_1 & P_2 \\ P_3 &
  P_4 \end{bsmallmatrix}$, $S=I_{n-m}$, $R=0$. By $*$-conjugating 
$\widetilde{A}$ with such $T$, we again keep zeros intact, change $\widetilde{H}$ into some
new Hermitian matrix, which we still call $\widetilde{H}$, and transform
$Y^*=\begin{bmatrix} I_k \quad L^*\end{bmatrix}$ into $\begin{bmatrix} P_1+L^*P_3 \quad P_2+L^*P_4\end{bmatrix}$. 
Let us see that we can make $P_1+L^*P_3$ and $P_2+L^*P_4=0$. By Autonne-Takagi theorem there exists $P_4$ such that
$P_4^T(I+L^TL)^*P_4=I$ and $P_1$ such that $P_1^T(I+LL^T)^*P_1=I$. We set $P_2=-L^*P_4$, 
$P_3=(L^T)^*P_1$ and thus we get
\begin{align*}
P^TP=                               \begin{bmatrix} P_1^T & P_1^TL^*\\
                                    -P_4^T(L^T)^*&P_4^T\end{bmatrix}\begin{bmatrix} P_1& -L^*P_4\\
                                    (L^T)^*P_1&P_4\end{bmatrix}
                                    = \begin{bmatrix}
                                    P_1^T(I+L^TL)^*P_1 & 0\\
                                    0&P_4^T(I+LL^T)^*P_4\end{bmatrix}=I. 
\end{align*} 
So $P$ is an $m\times m$ complex orthogonal matrix and $L^*$ is transformed to $0$, while $I_k$ in $Y^*$ stays unchanged. Hence we can assume $Y^*=\begin{bmatrix} I_k& 0\end{bmatrix}$.

We now choose $T$ in (\ref{stabB}) with $S=I_{n-m}$, $P=I_m$, $R=\begin{bsmallmatrix} 0\\
  R_2\end{bsmallmatrix}$, where $R_2=\begin{bmatrix} R_{21}&
  R_{22}\end{bmatrix}$ is a $k\times m$ matrix, with a $k\times
k$ matrix $R_{21}$ and a $k\times (m-k)$ matrix $R_{22}$, to be chosen
later. The conjugation by $T$ then only changes $\widetilde{H}$ in (\ref{formA2}), while
leaving everything else unchanged. By choosing an appropriate matrix $R_2$ the matrix $\widetilde{H}$ transforms into
\begin{equation}\label{formG}
\widetilde{H}+R_2^*Y^*+YR_2=\widetilde{H}+\begin{bmatrix} R_{21}^*+R_{21}& R_{22}\\
    R_{22}^*& 0\end{bmatrix}=
    \begin{bmatrix} 0& 0 \\
                                  0 & G\end{bmatrix},
    \end{equation}                           
where $G$ is a Hermitian $(m-k)\times (m-k)$ matrix. 
So, under an extra assumption that $I+L^TL$ and $I+LL^T$ are both nonsingular,
we can make the form of $\widetilde{A}$ to look like
\[\begin{bmatrix} 0& 0 & 0 & I_k\\
                                 0 & \mathcal{H}^{\epsilon}_{m-k} & 0 & 0\\
                                 0 & 0 & \I & 0 \\
                                 I_k & 0 &0 &0 
\end{bmatrix}.\]

If $I+L^TL$ and $I+LL^T$ are not both nonsingular,
we can still make the last change on $\widetilde{H}$ to get
\[\widetilde{A}=\begin{bmatrix} 0& 0 & 0 & I_k\\
                                 0 & \mathcal{H}^{\epsilon}_{m-k} & 0 & L\\
                                 0 & 0 & \I & 0 \\
                                 I_k & L &0 &0 
\end{bmatrix}.\]
Indeed, in this case $\widetilde{H}$ transforms to  
\(\widetilde{H}+\begin{bsmallmatrix} R_{21}^*+R_{21}& R_{22}+R_{21}^*L^*\\
    R_{22}^*+LR_{21} & R_{22}^*L^*+LR_{22}\end{bsmallmatrix},
    \) 
and by choosing a suitable matrices $R_{21}$, $R_{22}$, we get this matrix of the form (\ref{formG}).    
    
By $*$-conjugating with $I_k\oplus P\oplus I_{n-m}$, where $P$ is a suitable $(m-k)\times(m-k)$ complex orthogonal matrix, we finally achieve that by Hong's result \cite[Theorem 2.7]{Hong89} we can have $G=\mathcal{H}^{\epsilon}_{m-k}$ of the form (\ref{NF1}).
This completes the proof of the lemma.
\end{proof}

\begin{remark}
If the complex point is assumed to be nondegenerate, i.e. the corresponding matrix
$\begin{bsmallmatrix} A& \bar B\\B&\bar A\end{bsmallmatrix}$ is non-singular,
this means that $Y$ in Lemma \ref{prepare} must have rank $k$, and in particular $m\geq k$. 
\end{remark}

%Note that $(\mathcal{H}^{\epsilon}\oplus - 1,I_m\oplus 0)\sim (\mathcal{H}^{-\epsilon}\oplus 1,I_m\oplus 0)$.  
%Combining with the results presented in the preceding section this immediately implies:

\begin{theorem}\label{clasB}
A Hermitian-symmetric pair $(A,B)$ of $2\times 2$, $3\times 3$ and $4\times 4$ matrices is $\sim$-congruent to one of the following pairs $(\widetilde{A},\widetilde{B})$:\\
\begin{tabular}{l |l l}
\multicolumn{2}{c}{}\\
$\widetilde{A}$         & $\widetilde{B} $         &     \\

\hline
$\mathcal{I}_2$                          &    $0_2$   \\
\hline
%  &      &   \\
$\begin{bsmallmatrix}
0 & 1 \\
1 & 0
\end{bsmallmatrix}$   &      &      \\ 
%$a\oplus  0 $               &     &      \\  
$b \oplus 1$              &        &      \\ 
$a \oplus 0$              &   $1\oplus 0$        &      \\ 
%  &      &       \\
\hline                      
$\mathcal{H}^{\epsilon}_2$      &  $ I_2$\\
  &      &       \\
    &      &       \\
\multicolumn{2}{c}{}\\
$\widetilde{A}$         & $\widetilde{B} $         &     \\
\hline
$\mathcal{I}_3$                          &    $0_3$   \\
\hline
%  &      &       \\
$\begin{bsmallmatrix}  0 & 0 & 1 \\
0 & \epsilon & 0 \\
1 & 0 & 0
\end{bsmallmatrix}$\\
$b\oplus 1 \oplus 0$\\
$a\oplus 0_2 $ \\
$b\oplus \mathcal{I}_2 $               & $1\oplus 0_2$        &      \\
%  &      &       \\
 %                     
\hline
%  &      &       \\
$\begin{bsmallmatrix} 0 &0 &1\\
 0 &  a & 0\\
1 & 0 & 0
\end{bsmallmatrix}$\\
$\begin{bsmallmatrix} 0 &0 &1\\
 0 &  a & i\\
1 & -i & 0
\end{bsmallmatrix}$\\
$\mathcal{H}^{\epsilon}_2\oplus  0 $\\
$\mathcal{H}^{\epsilon}_2\oplus  1 $
                      & $I_2\oplus 0$        &      \\ 
   &      &       \\                     
\hline                      
$\mathcal{H}^{\epsilon}_3$      &  $ I_3$\\
  &      &       \\
\end{tabular} \qquad \qquad
\begin{tabular}{l |l l}
$\widetilde{A}$         & $\widetilde{B} $         &     \\

\hline
$\mathcal{I}_4$                       &    $0_4$   \\

\hline
$\begin{bsmallmatrix}
0 & 1 \\
1 & 0
\end{bsmallmatrix}\oplus 0_2$\\
$\begin{bsmallmatrix} 0 &0 & 1\\
 0 &  1 & 0\\
1 & 0 & 0
\end{bsmallmatrix}\oplus 0$\\
%$a\oplus 0_3 $             &        &      \\
$\begin{bsmallmatrix} 0 &0 & 1\\
 0 &  \mathcal{I}_2 & 0\\
1 & 0 & 0
\end{bsmallmatrix}$\\
$b\oplus \mathcal{I}_2\oplus 0 $\\   
$b\oplus 1 \oplus 0_2 $ \\   
$a\oplus 0_3 $ \\   
$b\oplus \mathcal{I}_3 $             & $1\oplus 0_3$        &      \\
\hline
$\begin{bsmallmatrix}
0 & 0 & 0 & 1\\
 0 &  b & 0 & \zeta \\
 0 & 0 & 1 & 0 \\
1 & \overline{\zeta} & 0 & 0
\end{bsmallmatrix}$\\
$\begin{bsmallmatrix}
0 & 0 & 1 \\
0 & a & 0\\
1 & 0 & 0
\end{bsmallmatrix}\oplus 0$\\
$\begin{bsmallmatrix}
0_2  & I_2 \\
I_2 & 0_2
\end{bsmallmatrix}$\\
$\mathcal{H}^{\epsilon}_2\oplus \mathcal{I}_2 $\\
$\mathcal{H}^{\epsilon}_2\oplus 1 \oplus 0 $\\
$\mathcal{H}^{\epsilon}_2\oplus 0_2 $
                      & $I_2\oplus 0_2$        &      \\                   
\hline
$\begin{bsmallmatrix}
0 & 0 & 0 & 1 \\
0 & b & \alpha & \zeta \\
0 & \overline{\alpha} & c & \xi\\
1 & \overline{\zeta} & \overline{\xi} & 0
\end{bsmallmatrix}$\\
$\mathcal{H}^{\epsilon}_3\oplus 1 $\\
$\mathcal{H}^{\epsilon}_3\oplus 0 $
                      & $I_3\oplus 0_1$        &      \\ 
                      
\hline                      
$\mathcal{H}^{\epsilon}_4$      &  $ I_4$
\end{tabular}\qquad \qquad\\
where $\epsilon\in \{0,1\}$, $a\in \mathbb{R}_{\geq 0}$, $b,c\in \mathbb{R}$, $\zeta,\xi\in \{0,i\}$, $\mathcal{H}^{\epsilon}_m$ is a $m\times m$ matrix of the form (\ref{NF1}), $\mathcal{I}_m$ for $m\in \{2,\ldots,4\}$ is any diagonal $m\times m$ matrix with only minus-ones, ones on the diagonal, and $\Trace (\mathcal{I}_m)\geq 0$. 
%Moreover, $(\mathcal{H}_m^{\epsilon},I_m)$ for $m\in \{2,3\}$ is determined up to the sign of $\epsilon$. 
\end{theorem}

\begin{proof}
The theorem is a consequence of Lemma \ref{prepare}. We only need to observe that some further simplifications can be done in the case when $Y$ in (\ref{NFp}) is a non-zero matrix, determined up-to right-multiplication with a non-singular matrix.

By $*$-conjugating with $\begin{bsmallmatrix} 1 & 0 \\
-\frac{a}{2} & 1
\end{bsmallmatrix}$, we obtain $\Big(\begin{bsmallmatrix} a & 1 \\
1 & 0
\end{bsmallmatrix},1\oplus 0\Big)\sim \Big(\begin{bsmallmatrix} 0 & 1 \\
1 & 0
\end{bsmallmatrix},1\oplus 0\Big)$.

We proceed with $3\times 3$ matrices. First, let $B=1\oplus 0_2$. Recall that any $3\times 3$ matrix $T$ for which
$T^TBT=B$ is of the form $T=\begin{bsmallmatrix} \pm 1 &0 \\
r & S
\end{bsmallmatrix},$
where $S\in GL_2(\mathbb C)$ and $r$ is a column vector in $\mathbb
C^2$ (see (\ref{stabB})). By choosing $S=I_2$ and $r^T=[0 \quad -\frac{a}{2}]$, then $*$-conjugating 
$A=\begin{bsmallmatrix}  a & 0 & 1 \\
0 & \epsilon & 0 \\
1 & 0 &  0
\end{bsmallmatrix}$, $\epsilon\in \{1,0\}$, with such $T$,
gives $\widetilde{A}=T^{*}AT=\begin{bsmallmatrix}  0 & 0 & 1 \\
0 & \epsilon & 0 \\
1 & 0 & 0
\end{bsmallmatrix}$.

Next, let $B=I_2\oplus 0$. Then any $3\times 3$ matrix $T$ such that $B$ is preserved after $*$-conjugating with $T$ is of the form 
\begin{equation}\label{T}
T=\begin{bmatrix} P &0 \\
r^T & s
\end{bmatrix},
\end{equation}
where $P$ is a $2 \times 2$ complex orthogonal matrix, $r$ is a column vector in $\mathbb
C^2$ and $s\in \mathbb C\setminus\{0\}$. By using an appropriate permutation for $P$ and a suitable value $s$ (with $r=0$) in (\ref{T}), the $*$-conjugation with such $T$ transforms the column vector $Y\in \mathbb{C}^2\setminus\{0\}$ of a Hermitian matrix 
\begin{equation}\label{HY}
A=\begin{bmatrix}  H & Y \\
Y^* & 0
\end{bmatrix} 
\end{equation}
into a vector $\begin{bsmallmatrix} 1 \\
b_2 \end{bsmallmatrix}\in\mathbb C^2$, and transforms $H$ into another
Hermitian matrix; we will still call them $Y$ and $H$ by a slight
abuse of notation.

If $b_2\ne \pm i$, we choose 
\[P^*=\tfrac{1}{\sqrt{1+b_2^2}}\begin{bmatrix} 1 & 
      b_2 \\
-b_2 & 1
\end{bmatrix}, \quad r=0,\quad s=\tfrac{1}{\sqrt{1+b_2^2}}\]
in (\ref{T}), to get the vector $Y=\begin{bsmallmatrix} 1 \\
b_2 \end{bsmallmatrix}$ in (\ref{HY}) into 
\[
\tfrac{1}{\sqrt{1+b_2^2}}\begin{bmatrix} 1 & 
      b_2 \\
-b_2 & 1
\end{bmatrix}\begin{bmatrix} 1 \\
b_2 \end{bmatrix}\begin{bmatrix}\tfrac{1}{\sqrt{1+b_2^2}}\end{bmatrix}=\begin{bmatrix} 1 \\
0 \end{bmatrix}.
\]
Let us further write $H=\begin{bsmallmatrix} h_{11}&h_{12} \\
\bar h_{12} & h_{22}\end{bsmallmatrix}$ in (\ref{HY}). The $*$-conjugation with $T$ of the form (\ref{T}) for
\[P=I_2 ,\quad r=\begin{bmatrix} -\frac{1}{2}h_{11} \\
-h_{12} \end{bmatrix},\quad s=1.\]
transforms $H$ to 
$\begin{bsmallmatrix} 0&0 \\
0 & h_{22}\end{bsmallmatrix},$
while not changing $Y$. By possibly multiplying the matrix by $-1$ and $*$-conjugating it by $1\oplus h_{22}^{-1}\oplus -1$ (if $h_{22}\neq 0$), we achieve that $h_{22}\in \{0,1\}$. When $b_2=\pm i$, then by a similar choice of transformations, we can first 
get $H$ into the form $\begin{bsmallmatrix} 0&h_{12} \\
\bar{h}_{12} & h_{22}\end{bsmallmatrix}$ and then obtain $\begin{bsmallmatrix} 0&0 \\
0 & h_{22}\end{bsmallmatrix}$, with $h_{22}\geq 0$. After a possible $*$-conjugation by $1\oplus -1\oplus 1$, we can assume $b_2=i$.
Hence, in the case when $A$ is of the form (\ref{HY}) with a non-zero $Y$ and $B=\widetilde{B}=I_2\oplus 0$, we may assume that $\widetilde{A}$ is any of the matrices
\[
\begin{bmatrix} 0 &0 &1\\
 0 &  a & 0\\
1 & 0 & 0
\end{bmatrix},\quad
\begin{bmatrix} 0 &0 &1\\
 0 &  a & i\\
1 & -i & 0
\end{bmatrix}, \qquad a\in \mathbb{R}_{> 0}.
\]

Finally, we consider $4\times 4$ matrices. In case $B=1\oplus 0_3$ we obtain the following possible non-diagonal normal forms for $\widetilde{A}$:\\
\[
\begin{bmatrix}
a & 1 \\
1 & 0
\end{bmatrix}
\oplus 0_2,\,\begin{bmatrix} a &0 & 1\\
 0 &  1 & 0\\
1 & 0 & 0
\end{bmatrix}\oplus 0,\,
\begin{bmatrix} a &0 & 1\\
 0 &  \mathcal{I}_2 & 0\\
1 & 0 & 0
\end{bmatrix},\quad a\in \mathbb{R}_{\ge 0},
%b\oplus \mathcal{I}_2\oplus 0 ,b\oplus 1\oplus 0_2, a\oplus 0_3 ,b\oplus \mathcal{I}_3, 
\]
for $B=I_2\oplus 0_2$ we get
\[
\begin{bmatrix} \mathcal{H}^{\epsilon}_2 &0 & Z\\
 0 &  1 & 0\\
Z^{*} & 0 & 0
\end{bmatrix},\,
\begin{bmatrix}
\mathcal{H}^{\epsilon}_2 & I\oplus 0 \\
I\oplus 0 & 0_2
\end{bmatrix},\,
\begin{bmatrix}
\mathcal{H}^{\epsilon}_2  & I_2 \\
I_2 & 0_2
\end{bmatrix}, \quad Z\in \mathbb{C}^2\setminus\{0\},
%\mathcal{H}^{\epsilon}_2\oplus \mathcal{I}_2,\,\mathcal{H}^{\epsilon}_2\oplus 0_2 , \mathcal{H}^{\epsilon}_2\oplus 1 \oplus 0,
\]
and when $I_3\oplus 0_1$ we have 
$\begin{bsmallmatrix}
\mathcal{H}^{\epsilon}_3 & Y \\
Y^{*} & 0_1
\end{bsmallmatrix}
%\mathcal{H}^{\epsilon}_3\oplus 1,\, 
%\mathcal{H}^{\epsilon}_3\oplus 0 .
$ with $Y\in
\mathbb{C}^3\setminus\{0\}$;
here $\mathcal{H}^{\epsilon}_m$ is a $m\times m$ matrix of the form
(\ref{NF1}), $\mathcal{I}_m$ for $m\in \{2,3,4\}$ is any diagonal
$m\times m$ matrix with only minus-ones, ones on the diagonal, with $\Trace (\mathcal{I}_m)\geq 0$. In a similar fashion as in lower dimensions we now perform further simplifications on these forms.
\end{proof}

If dimensions $\leq 4$ we can also tell something about uniqueness of the normal form of $(A,I)$.

\begin{trditev}\label{unique}
Suppose $A$ is a $n\times n$ Hermitian matrix and $n\leq 4$. Then
$(A,I_n)\sim (\mathcal{H}^{\epsilon_1}(A),I_n)\sim
(\mathcal{H}^{\epsilon_2}(A),I_n)$ if and only if $\epsilon_1=\pm
\epsilon_2$, except for $(H_3(0)\oplus x,I_4)\sim (-H_3(0)\oplus
x,I_4)$, $x>0$. Moreover, the only orthogonally $*$-congruent pair of matrices of the form (\ref{NF1}) of dimension $\leq 3$ is $H_3(0)$, $-H_3(0)$. %($*$-conjugated by $-1\oplus 1\oplus -1$).
\end{trditev}

\begin{proof}
For $2\times 2$, $3\times 3$ and $4\times 4$ matrices we do a direct computation to find out when $(\mathcal{H}^{\epsilon_1}(A),I)\sim(\mathcal{H}^{\epsilon_2}(A),I)$, where  $\mathcal{H}^{\epsilon_2}(A),\mathcal{H}^{\epsilon_2}(A)$ of the form (\ref{NF1}) are Hermitian canonical forms of $A$. More precisely, we solve the matrix equation $\varepsilon Q^{*}(\mathcal{H}^{\epsilon_1}(A))Q=\mathcal{H}^{\epsilon_2}(A)$ or the equivalent equation  $(\mathcal{H}^{\varepsilon \epsilon_1}(A))Q= \overline{Q}(\mathcal{H}^{\epsilon_2}(A))$ (see the exposition of the normal FORM 1 in Section \ref{sec1}); here $\varepsilon \in \{1,-1\}$ and $Q$ is a complex orthogonal matrix to be determined. Recall also that by Sylvester's theorem the inertia matrix is an invariant under $*$-conjugation with a non-singular matrix.

If the canonical forms of $A$ are diagonal matrices, it is in dimension $4$ sufficient to consider the pair $\mathcal{G}=\lambda\oplus \mu \oplus \nu\oplus\eta$, $\mathcal{H}=-\lambda\oplus -\mu \oplus\nu\oplus \eta$, where $\nu,\eta\in \mathbb{R}$ and  $\lambda,\mu\in \mathbb{R}\setminus \{0\}$ are of the opposite signs, with $|\lambda|\neq |\mu|$. Comparing the entries, their moduli or the real parts in the matrix equation $\mathcal{G}Q=\overline{Q}\mathcal{H}$, $Q=[q_{jk}]_{j,k}^3$, then implies: 
\begin{align*}
&0=\lambda (q_{11}+\overline{q}_{11}),\quad 0=\mu (q_{22}+\overline{q}_{22}),\\
&0=|q_{21}|(|\lambda|-|\mu|)=|q_{12}|(|\lambda|-|\mu|), \qquad \\
&0=|q_{31}|(|\lambda|-|\nu|)=|q_{13}|(|\lambda|-|\nu|),\qquad \\
&0=|q_{41}|(|\lambda|-|\eta|)=|q_{14}|(|\lambda|-|\eta|),\qquad \\
&0=|q_{32}|(|\mu|-|\nu|)=|q_{23}|(|\mu|-|\nu|) \\
&0=|q_{42}|(|\mu|-|\eta|)=|q_{24}|(|\mu|-|\eta|) \\
%&0=|q_{32}|(|\nu|-|\eta|)=|q_{23}|(|\nu|-|\eta|),\\
&\lambda \Rea (q_{31})=\nu \Rea(q_{31}), \quad -\lambda \Rea (q_{13})=\nu \Rea(q_{13}),\\
%&\lambda \Rea (q_{41})=\eta \Rea(q_{41}), \quad -\lambda \Rea (q_{14})=\eta \Rea(q_{14}),\\
&\mu \Rea (q_{32})=\nu \Rea(q_{32}), \quad -\mu \Rea (q_{23})=\nu \Rea(q_{23}),\\
%&\mu \Rea (q_{42})=\eta \Rea(q_{42}), \quad -\mu \Rea (q_{24})=\eta \Rea(q_{24}),\\
\end{align*}
It is immediate that $\Rea (q_{11})=\Rea (q_{22})=0$, $q_{12}=q_{21}=0$. Since any of $\nu,\eta$ is different from at least one of the constants $\lambda,\mu$, we have either $q_{23}=q_{32}=0$ or $q_{13}=q_{31}=0$ (or both) and either $q_{24}=q_{42}=0$ or $q_{14}=q_{41}=0$ (or both). Furthermore, at least one of the entries $q_{13}$, $q_{31}$ (and $q_{23}$, $q_{32}$) must be purely imaginary. It follows that at least one of the first two columns or one of the first two rows of $Q$ is purely imaginary, 
%$q_{1j}^2+q_{2j}^2+q_{3j}^2=-(\Ima (q_{1j}))^2\neq 1$ for some $j\in \{1,2\}$. 
%This means that at least one of the first two columns is not normalized,
hence $Q$ cannot be complex orthogonal. In the case of $3\times 3$ (or $2\times 2$) matrices it is enough to check the pair $\lambda\oplus \mu \oplus\nu$ and $-\lambda\oplus -\mu \oplus\nu$ (the pair $\lambda\oplus \mu $ and $-\lambda\oplus -\mu $) with $\lambda,\mu\in \mathbb{R}\setminus \{0\}$ of the opposite signs and different moduli, with $\nu\in \mathbb{R}$. By similar computations as in the $4$-dimensional case, the orthogonal $*$-congruence fails again.

Next, we take the pair $H_2(x)\oplus \lambda\oplus-\mu$ and $H_2(x)\oplus -\lambda\oplus \mu$, with $x,\lambda,\nu\in \mathbb{R}_{> 0}$, $\lambda\neq \mu$.
Let now $Q$ be of the form 
$\begin{bsmallmatrix}
Q_1 & Q_2 \\
Q_3 & Q_4
\end{bsmallmatrix}$, where $Q_1,Q_2,Q_3,Q_4$ are all $2\times 2 $ matrices. The matrix equation $(H_2(x)\oplus \lambda\oplus-\mu)Q=\overline{Q}(H_2(x)\oplus -\lambda\oplus \mu)$ then yields
\begin{align*}
&H_2(x)Q_1=\overline{Q}_1H_2(x), \quad (\lambda\oplus -\mu)Q_4=\overline{Q}_4 (-\lambda\oplus \mu)\\
&H_2(x)Q_2=\overline{Q}_2(-\lambda\oplus \mu),\quad  (\lambda\oplus \mu)Q_3=\overline{Q}_3H_2(x).
\end{align*}
We compute $Q_2$, $Q_3$ and observe that the columns of $Q_2$ and the rows of $Q_3$ are of the form $s\begin{bsmallmatrix}
1-i\\
i+1
\end{bsmallmatrix}$, $s\in \mathbb{R}\cup i\mathbb{R}$ (see also Lemma \ref{ConForm}).
Also, we get that $Q_4$ is of the form $\begin{bsmallmatrix}
iu & 0\\
0 &  iv
\end{bsmallmatrix}$, $u,v\in \mathbb{R}$.
It is now easy to see that the last two columns of the matrix $Q$ are
not normalized and therefore $Q$ cannot be complex orthogonal.

We recall that $H_3(0)$ and $-H_3(0)$ are orthogonally $*$-congruent
($*$-conjugated by $-1\oplus 1\oplus -1$), while $H_3(x)$, $-H_3(x)$
for $x>0$ have different inertia matrices, so 
$(H_3(0)\oplus x,I_4)\sim (H_3(0)\oplus x,I_4)$ follows.

To conclude the proof of the second part of the statement of the lemma, it is left to consider only a few matrices paired with its additive inverses.
We first consider the pair $H_2(0)\oplus -\lambda$, $-H_2(0)\oplus \lambda$ with $\lambda\in \mathbb{R}_{>0}$. Suppose $Q$ is of the form 
$\begin{bsmallmatrix}
Q' & b \\
c^T & s
\end{bsmallmatrix}$, where $Q'$ is a $2\times 2 $ matrix, $b,c$ are column vectors in $\mathbb{C}^2$ and $s\in \mathbb{C}$. The matrix equation $(H_2(0)\oplus -\lambda)Q=\overline{Q}(-H_2(x)\oplus \lambda)$ implies
\[
\begin{bmatrix}
H_2(0)Q' & H_2(0)b\\
-\lambda c^T & -\lambda s
\end{bmatrix}
=
\begin{bmatrix}
-\overline{Q'}H_2(0) & \lambda \overline{b}\\
-\overline{c}H_2(0) & -\lambda \overline{s}
\end{bmatrix}.
\]
From $H_2(0)b=\lambda \overline{b}$ and $-\lambda s=-\lambda \overline{s}$ it then follows that $b=0$ and $s=it$ for some $t\in \mathbb{R}$ (remember $\lambda\neq 0$). Hence $b^Tb+s^2<0$ and the last column of $Q$ is not normalized ($Q$ is not complex orthogonal).

We proceed with $H_2(x)\oplus 0$ and $-H_2(x)\oplus 0$, $x\in \mathbb{R}_{> 0}$ (hence $H_2(x)$ and $-H_2(x)$) are not orthogonally $*$-congruent. By multiplying the matrices as block matrices in the equation $(H_2(x)\oplus 0)Q=-\overline{Q}(H_2(x)\oplus 0)$ immediately yields $H_2(x)[q_{13} \quad q_{23}]^T=0$, $-[q_{31}\quad q_{32}]H_2(z)=0$. Since $H_2(x)$ is non-singular, we have $q_{13}=q_{23}=q_{31}=q_{32}=0$. It follows that $Q=Q'\oplus q_{33}$, where $Q'=[q_{j,k}]_{j,k=1}^2$ is a complex orthogonal matrix and $H_2(x)Q'=\overline{Q'}-H_2(x)$. By comparing the entries in this matrix equation, regrouping the like-terms, and slightly simplifying, we get
\begin{align}\label{sistem}
 &\Rea(q_{11})=-q_{21}(x+\tfrac{i}{2})-\overline{q}_{12}(x-\tfrac{i}{2}),\quad & \Rea(q_{22})=-q_{12}(x-\tfrac{i}{2})-\overline{q}_{21}(x+\tfrac{i}{2})\\
 &\Rea(q_{12})=-(x+\tfrac{i}{2})(\overline{q}_{11}+q_{22}), \quad & \Rea(q_{21})=-(x-\tfrac{i}{2})(\overline{q}_{22}+q_{11}),\qquad    \nonumber
\end{align}
Trivially, the second pair of the above equations (\ref{sistem}) implies that
\begin{align*}
&0=-\tfrac{1}{2}(\Rea (q_{11})+\Rea(q_{22}))-z(\Ima (q_{22})-\Ima (q_{11}))\\
&\Rea (q_{12})=\Rea(q_{21})=-z(\Rea (q_{11})+\Rea(q_{22}))+\tfrac{1}{2}(\Ima(q_{22})-\Ima (q_{11})),
\end{align*}
and it is immediate that $z\Rea (q_{12})=-(z^2+\tfrac{1}{4})(\Rea (q_{11})+\Rea(q_{22}))$.
Adding 
%and subtracting 
the first two equations in (\ref{sistem}), regrouping the like-terms and combining with the above equations yields:
\begin{align*}
\Rea (q_{11})+\Rea(q_{22})&=%-\Rea (q_{12})(2z-i)-\Rea (q_{21})(2z+i)=
-2z(\Rea (q_{12})+\Rea(q_{21}))+i(\Rea (q_{12})-\Rea(q_{21}))\\
                          &=-4z\Rea (q_{12})=(4z^2+1)(\Rea (q_{11})+\Rea(q_{22})),\\
\Rea (q_{11})-\Rea(q_{22})&=-(\Ima (q_{21})+\Ima (q_{12}))+2iz(\Ima (q_{21})-\Ima (q_{12})).
\end{align*}
It follows that $\Rea (q_{11})+\Rea(q_{22})=0$ and $\Rea (q_{11})-\Rea(q_{22})=-(\Ima (q_{21})+\Ima (q_{12}))$, therefore we have $\Rea (q_{12})=\Rea(q_{21})=\tfrac{1}{2}(\Ima(q_{22})-\Ima (q_{11}))$ and thus 
$\Rea (q_{11})=-\Rea(q_{22})=-\frac{1}{2}(\Ima (q_{21})+\Ima (q_{12}))$.
Hence $Q'$ is of the form
\[
Q'=\begin{bmatrix}
-\frac{1}{2}(b+c)+ia & \frac{1}{2}(d-a)+ic \\
\frac{1}{2}(d-a)+ib & \frac{1}{2}(b+c)+id
\end{bmatrix}, \qquad a,b,c,d\in \mathbb{R}.
\]
A necessary condition for $Q'$ to be complex orthogonal is that the anti-diagonal entries and the difference of the diagonal entries of $Q'^TQ'$ vanish, i.e. $a^2-b^2+c^2-d^2-2i(ac+bd)=0$ and $-a^2-b^2+c^2+d^2-2i(ab+cd)=0$. The later implies that $a^2=d^2$, $c^2=d^2$, $ac=-bd$. It is easy and straightforward to verify that in any of these cases we obtain the contradiction with the orthogonality of $Q'$.
\end{proof}

\begin{remark}  We expect that a similar uniqueness result is also true in
  dimensions $\ge 5$, but we must at least take into account that  for any $n$,
  $H_{2n-1}(0)$ is $*$-congruent by the complex orthogonal
  matrix $Q=-1\oplus 1\oplus -1\oplus
  \cdots\oplus -1$ to $-H_{2n-1}(0)$. 
\end{remark}

Next, we try to find a normal form $(A,B)$ with a nice matrix $A$. Using Proposition \ref{clasB} one can 
transform the pairs of the matrices $(\widetilde{A},\widetilde{B})$ in the table of the proposition to the 
form $(\mathcal{I}(\widetilde{A}),\mathcal{N}_{\widetilde{B}}(\widetilde{A}))$ (see (\ref{inertia}), (\ref{NBA})). This is an easy and 
straightforward computation for the block-diagonal matrices with blocks of dimensions less than $3$, 
while if blocks are of dimension $3$ or $4$ the calculation is more tedious due to solving the cubic or quartic equation.

\begin{trditev}\label{clasA}
A pair $(A,B)$ of a $2\times 2$ (or $3\times 3$) Hermitian matrix $A$ and a symmetric matrix $B$ is $\sim$-congruent to one of the following pairs 
$(\widetilde{A},\widetilde{B})$:\\
\begin{tabular}{l |l l}
$\widetilde{A}$         & $\widetilde{B} $         &     \\
\hline
$-1\oplus 1$ 
 &    $\tfrac{1}{2x-i}
\begin{bsmallmatrix}
\frac{4x}{w-1} & -\frac{i}{x}\\
-\frac{i}{x} & \frac{4x}{w+1}
\end{bsmallmatrix}$ & $x>0$, $w=\sqrt{1+4x^2}$  \\
     & $\tfrac{1}{2\overline{\xi}|\xi|}
\begin{bsmallmatrix}
\xi+\overline{\xi} & \overline{\xi}-\xi\\
\overline{\xi}-\xi & \xi+\overline{\xi}
\end{bsmallmatrix} $       &  $\xi\in \mathbb{C}\setminus\mathbb{R}$  \\ 
%
%\hline 
%$-1\oplus 1 $  
   &     
$\begin{bsmallmatrix}
1  & -1 \\
-1 & 1
\end{bsmallmatrix}$   &  \\  
%
%\hline
%$-1\oplus 1 $  
        & $x \oplus y  $, \quad $x \oplus 0  $ , \quad $0_2$         &  $x,y> 0$    \\ 
\hline
$I_2$  & $x \oplus y  $, $x \oplus  0  $ , $0_2$       &  $x,y> 0$   \\ 
\hline
$0_2$  & $0_2 $, $1 \oplus 0_1  $, $I_2  $       &    \\ 

\hline
$1\oplus 0$  & 
                $0_2 $, $x \oplus 0  $ , $x\oplus 1  $, $0\oplus 1 $,       &  $x> 0$   \\ 
               &   
               $
\begin{bsmallmatrix}
0 & 1\\
1 & 0
\end{bsmallmatrix}$  &  \\
                 
\end{tabular}

\textrm{ } \\
%
%
%
%
%$(\widetilde{A},\widetilde{B})$:
\\
\begin{tabular}{l |l l}
$\widetilde{A}$         & $\widetilde{B} $         &     \\ 
\hline           
%
%
%\hline
     $-1\oplus 1\oplus 0$ 
               & $\frac{1}{2}
\begin{bsmallmatrix}
i & -i & \sqrt{2}\\
-i & i & \sqrt{2}\\
\sqrt{2} & \sqrt{2} & 0
\end{bsmallmatrix}$        &     \\ 
%
%\hline
%$-1\oplus 1\oplus 0$ 
%     & $\tfrac{1}{2\overline{\xi}|\xi|}\left[
%\begin{array}{cc}
%\xi+\overline{\xi} & \overline{\xi}-\xi\\
%\overline{\xi}-\xi & \xi+\overline{\xi}
%\end{array}
%\right] \oplus x $       &   \\ 
%
%\hline
%$-1\oplus 1\oplus 0$ 
    & $\tfrac{1}{2\overline{\xi}|\xi|}
\begin{bsmallmatrix}
\xi+\overline{\xi} & \overline{\xi}-\xi\\
\overline{\xi}-\xi & \xi+\overline{\xi}
\end{bsmallmatrix}
 \oplus \varepsilon $       &   $\xi\in \mathbb{C}\setminus\mathbb{R}$, \, $\varepsilon \in \{0,1\}$  \\ 
%
%\hline 
%$-1\oplus 1\oplus 0$  
%   &      $
%\left[\begin{array}{c c}
%1  & -i \\
%-i & 1
%\end{array}
%\right]\oplus 0_1$ ,  &   \\  

%\hline
%$-1\oplus 1\oplus 0$  
        & $x \oplus y \oplus 1  $, \quad $x \oplus y \oplus 0  $       &  $x,y> 0$    \\ 
     &       $x \oplus 0 \oplus y  $, \quad $0_2 \oplus 1$, \quad $0_3$         &  $x,y> 0$    \\ 
        &  $ x\oplus 
\begin{bsmallmatrix}
0 & 1\\
1 & 0
\end{bsmallmatrix}$     &         $x>0$ \\
      &      $\tfrac{1}{2x-i}
\begin{bsmallmatrix}
\frac{4x}{w-1} & -\frac{i}{x}\\
-\frac{i}{x} & \frac{4x}{w+1}
\end{bsmallmatrix}\oplus \varepsilon $                         &  $w=|2x+i|$,\, $x>0$, \, $\varepsilon \in \{0,1\}$   \\  
%$-1\oplus 1\oplus 0$  
&  
$\frac{1}{2}\begin{bsmallmatrix}  1 & -1 \\
-1 & 1  \\
\end{bsmallmatrix} \oplus \varepsilon $   &    $\varepsilon \in \{0,1\}$ \\
%&  
%$\frac{1}{2}\begin{bmatrix}  1 & -1  \\
%-1 & 1 \\
%\end{bmatrix}\oplus 1$   &    \\
% 
 &  
 $
\tfrac{1}{\sqrt[4]{8}}
\begin{bsmallmatrix}
0 & 0 & i \\
0 & 0 & -i\\
i  & -i & 0
\end{bsmallmatrix}$ &   \\
\hline
  $-1\oplus I_2$ 
   & $\left[u(\lambda_j,\lambda_k)\right]_{j,k=1}^3$  &    $ a>0$, \quad $\lambda_{1}\leq\lambda_{2},\lambda_{3}$\\
   &                                                  & ${\scriptstyle -\lambda_j^3+a\lambda_j^2+2\lambda_j-a=0,\,\,j=1,2,3}$\\
   &                                                  &
$u(\lambda,\mu)=\tfrac{\sign (\lambda)\sign (\mu)(\lambda^2+\mu^2-\lambda^2\mu^2)}{\sqrt{(\lambda^4-\lambda^2+2)(\mu^4-\mu^2+2)}}$\\         
%
%  $-1\oplus 1\oplus \epsilon$ 
%\hline
     %$-1\oplus 1\oplus \epsilon$ 
               & $-i[p(\lambda_j,\lambda_k)q(\lambda_j)q(\lambda_k)]_{j,k=1}^3$ &  ${\scriptstyle p(\lambda,\mu)=\lambda^2\mu^2-x^2(\lambda^2+\mu^2)+x(\lambda+\mu)+x^4}$ \\
               &                                                &   ${\scriptstyle q(\lambda)=\tfrac{\sqrt{\lambda^2+x^2}}{(x-i\lambda)\sqrt{|\lambda|(\lambda^4+\lambda^2(1-2x^2)+x^2+x^4)}}}$   \\ 
                                                             &                  &   $\lambda_{1}\leq \lambda_{2}\leq\lambda_{3}\text{ eigenvalues of }H_3(x)$, $x> 0$ \\
%
%
%\hline
%$-1\oplus 1\oplus 1$ 
%     & $\tfrac{1}{2\overline{\xi}|\xi|}\left[
%\begin{array}{cc}
%\xi+\overline{\xi} & \overline{\xi}-\xi\\
%\overline{\xi}-\xi & \xi+\overline{\xi}
%\end{array}
%\right] \oplus 0 $       &   $\xi\in \mathbb{C}\setminus\mathbb{R}$   \\ 
%
%\hline
%$-1\oplus 1\oplus \epsilon$ 
    & $\tfrac{1}{2\overline{\xi}|\xi|}
\begin{bsmallmatrix}
\xi+\overline{\xi} & \overline{\xi}-\xi\\
\overline{\xi}-\xi & \xi+\overline{\xi}
\end{bsmallmatrix} \oplus x $       &    $\xi\in \mathbb{C}\setminus\mathbb{R}$,\, $x\geq 0$     \\ 
 
%\hline
%$-1\oplus 1\oplus \epsilon$  
        &  $x \oplus 0_2$, $0_2\oplus x$                                             &  $x\geq 0$   \\
        &  $x \oplus y \oplus 0  $,   $0 \oplus x \oplus y  $         &  $x,y> 0$    \\ 
        &   $x \oplus y \oplus z  $, \quad $0_3$                                                  & $x,y,z>0$     \\
%        &    $\tfrac{1}{2x-i}\left[
%\begin{array}{cc}
%\frac{4x}{w-1} & -\frac{i}{x}\\
%-\frac{i}{x} & \frac{4x}{w+1}
%\end{array}
%\right]\oplus x$           &    $w=|2x+i|$, $x\in \mathbb{R}^+$, \, $x\geq 0$    \\ 
      &      $\tfrac{1}{2x-i}
\begin{bsmallmatrix}
\frac{4x}{w-1} & -\frac{i}{x}\\
-\frac{i}{x} & \frac{4x}{w+1}
\end{bsmallmatrix}\oplus y$           &       $w=|2x+i|$, $y\geq 0$,     \\
&  
$\frac{1}{2}\begin{bsmallmatrix}  1 & -1  \\
-1 & 1  \\
\end{bsmallmatrix}\oplus x $   &  $x \geq 0$  \\
%&  
%$\frac{1}{2}\begin{bmatrix}  1 & -1 & 0 \\
%-1 & 1 & 0 \\
%0 & 0 & 2
%\end{bmatrix}$  & \\
%
\hline
$I_3$  & $x \oplus y \oplus z  $, $x \oplus y \oplus 0  $    &  $x,y,z> 0$   \\ 
       &  $x \oplus 0_2  $                            &  $x\geq 0$   \\
\hline
$0_3$  & $0_3 $, $1 \oplus 0_2  $, $I_2 \oplus 0 $ , $I_3$       &    \\ 

\hline
$ I_2\oplus 0$  & $0_3 $, $0_2\oplus 1$, $x \oplus 0_2  $ , $x\oplus 0\oplus 1  $       &  $x> 0$   \\ 
                & $x\oplus y\oplus 0  $ , $x\oplus y \oplus 1$        &  $x,y>0$   \\
                   &    $ x\oplus  
\begin{bsmallmatrix}
0 & 1\\
1 & 0
\end{bsmallmatrix} $      &  $x\geq 0$  \\
\hline
$1 \oplus 0_2$  & $0_3 $ , $0\oplus I_2$, $0\oplus 1 \oplus 0$       &     \\ 
                &   $x \oplus 0_2  $, $x\oplus 1\oplus 0  $ , $x\oplus I_2  $            & $x>0$  \\
                &  $ 
\begin{bsmallmatrix}
0 & 1\\
1 & 0
\end{bsmallmatrix}
\oplus \varepsilon $        &    $\varepsilon\in \{0,1\}$        \\

%\hline
%$1\oplus -1\oplus -1$  & $\mathcal{N}_{I_3}(-H_3(z))$        &  $z>0$   \\ 

%\hline
%$1\oplus -1\oplus -1$  & $\mathcal{N}_{I_2\oplus 0_1}(M(a,b,\zeta))$        &  $a|\zeta|^2+b(|ab|+1)<0$   \\ 
                      
\end{tabular}
\end{trditev}

\begin{proof}
We can easily see that
$
\Big(\begin{bsmallmatrix}
0 & 1 \\
1 & 0
\end{bsmallmatrix}                   , 1\oplus 0\Big)\sim 
\Big(-1\oplus 1,\begin{bsmallmatrix}
                  1  & -i \\
                   -i & 1
                 \end{bsmallmatrix}
                             \Big)
                              $ and
$
\left(\begin{bsmallmatrix}
0 & 0 & 1 \\
0& \epsilon & 0 \\
1 & 0 & 0
\end{bsmallmatrix}                   , 1\oplus 0_2\right)\sim 
\left(-1\oplus 1\oplus \epsilon,\frac{1}{2}\begin{bsmallmatrix}  1 & -1 & 0 \\
-1 & 1 & 0 \\
0 & 0 & 0
\end{bsmallmatrix}\right)$, $\epsilon\in \{0,1\}.$

Next, we examine $M(a,\zeta)=
\begin{bsmallmatrix}
0 & 0 & 1\\
0 & a & \zeta\\
1 & \overline{\zeta} & 0
\end{bsmallmatrix}$, $a\geq 0$, $\zeta\in \{0,i\}$. By inspecting the signs of the coefficients of 
\[
\Delta_{M(a,\zeta)}(\lambda)=-\lambda^3+a\lambda^2+(1+|\zeta|^2)\lambda-a,
\]
we obtain that signs of the eigenvalues of $M(a,\zeta)$ are $-1$, $1$, $\sign (a)$. 
%Since $1+|\zeta|^2>0$, one eigenvalue is positive and one is negative, and the third one is positive for $a>0$ and zero if $a=0$. 
%Further, we see that $a$ is the eigenvalue of $M(a,\zeta,\delta)$ if and only if $a(|\zeta|^2+\delta^2-1)=0$, and this case  $\Delta_{M(a,\zeta)}(\lambda)=-(a-\lambda)(|\zeta|^2+1-\lambda^2)$ with zeros $\lambda_0=a$, $\lambda_{1,2}=\sqrt{|\zeta|^2+1}$. We obtain 
%\[
%\mathcal{N}_{I_2\oplus 0}(M(a,\zeta))=
%\left[
%\begin{array}{c c c}
%ts(\lambda_2,\lambda_2) & -r(\lambda_2) & -ts(\lambda_1,\lambda_2)\\
%-r(\lambda_2) & \tfrac{t}{b^{2\sign(b)}} & r(\lambda_1)\\
%-ts(\lambda_1,\lambda_2) & r(\lambda_1) & ts(\lambda_1,\lambda_1)
%\end{array}
%\right],
%\]
%where 
%\[
%t=\tfrac{\zeta^2+1}{|\zeta|^2+1}, \quad r(\lambda)=\tfrac{2i\Ima(\zeta)}{(|\zeta|^2+1)\sqrt{a^{2\sign(a)}|2\lambda-a|}},\quad
%s(\lambda,\mu)=|(2\lambda-a)(2\mu-a)|^{-\tfrac{1}{2}}.
%\]
If $a>0$, $\zeta=i$ it is somewhat tedious but still straightforward to compute that
\[
\mathcal{N}_{I_2\oplus 0}(M(a,i))=\left[u(\lambda_j,\lambda_k)\right]_{j,k=1}^3,\quad a>0,
\]
where $\lambda_1,\lambda_2,\lambda_3$ are the eigenvalues of $M(a,i)$ with $\lambda_1$ negative, and
\[
u(\lambda,\mu)=\tfrac{\sign (\lambda)\sign (\mu)(\lambda^2+\mu^2-\lambda^2\mu^2)}{\sqrt{(\lambda^4-\lambda^2+2)(\mu^4-\mu^2+2)}}.
\]
Otherwise we easily get
\begin{align*}
&\mathcal{N}_{I_2\oplus 0}(M(a,0))=
\tfrac{1}{2}
\begin{bmatrix}
1 & -1 \\
-1 & 1 \\
\end{bmatrix}\oplus \widetilde{a}, \quad \widetilde{a}=\left\{
\begin{array}{ll}
a^{-1}, & a> 0 \\
1,         & a=0
\end{array}\right.,\\
&\mathcal{N}_{I_2\oplus 0}(M(0,i))=
\tfrac{1}{\sqrt[4]{8}}
\begin{bmatrix}
0 & 0 & i \\
0 & 0 & -i\\
i  & -i & 0
\end{bmatrix}.
\end{align*}

It remains to consider $(\mathcal{H}^{\epsilon}_m,I_m)\sim (\mathcal{I}(\mathcal{H}^{\epsilon}_m),\mathcal{N}_{I_m}(\mathcal{H}^{\epsilon}_m))$ for $m\in \{1,2,3\}$ and  $(\mathcal{H}^{\epsilon}_r\oplus \mathcal{I}_s,I_r\oplus 0_s)\sim (\mathcal{I}(\mathcal{H}^{\epsilon}_r)\oplus \mathcal{I}_s,\mathcal{N}_{I_r}(\mathcal{H}^{\epsilon}_r)\oplus 0_s)$ for $r,s\in \{1,2\}$ with $r+s\in \{2,3\}$). Conveniently, in the preceding section we have calculated all $\mathcal{I}(\mathcal{H})$ and $\mathcal{N}_{I_m}(\mathcal{H})$ if $\mathcal{H}$ was any of the matrices $H_2(x)$, $H_3(x)$, $K_1(\xi)$ (see (\ref{Hmz}), (\ref{KLm})); remember that $K_1(y)=L_1(-iy)$.

Clearly, $(a\oplus 0,1\oplus 0)\sim (1\oplus 0,\frac{1}{a}\oplus 0)$ for $a>0$, while $(b\oplus 1,1\oplus 0)$, $b\in \mathbb{R}$, splits into classes of three types: $(1\oplus 0,0\oplus 1)$, $(-1\oplus 1,x\oplus 0)$, $(1\oplus 1,x\oplus 0)$, $x>0$. Further, from $(\mathcal{H}_2^{\epsilon},I_2)$ we obtain $(1\oplus 0,x\oplus 1)$, $\big(-1\oplus 1,\tfrac{1}{2\overline{\xi}|\xi|}
\begin{bsmallmatrix}
\xi+\overline{\xi} & \overline{\xi}-\xi\\
\overline{\xi}-\xi & \xi+\overline{\xi}
\end{bsmallmatrix}\big)$, $(-1\oplus 1,x\oplus y)$, $(1\oplus 1,x\oplus y)$, $(0_2,I_2)$,  $\big(-1\oplus 1, \tfrac{1}{2x-i}
\begin{bsmallmatrix}
\frac{4x}{w-1} & -\frac{i}{x}\\
-\frac{i}{x} & \frac{4x}{w+1}
\end{bsmallmatrix}\big)$,  $(1\oplus 0, 
\begin{bsmallmatrix}
0 & 1\\
1 & 0
\end{bsmallmatrix}$), where $w=\sqrt{1+4x^2}$, $x,y>0$, $\xi \in \mathbb{C}\setminus{R}$. By adding a direct summand $0$ to these pairs of matrices we obtain pairs that are $\sim$-congruent to $(a\oplus 0_2,1\oplus 0_2)$, $(b\oplus 1\oplus 0,1\oplus 0_2)$, $(\mathcal{H}_2^{\epsilon}\oplus 0,I_2\oplus 0)$.

Similarly we deal with $(\mathcal{H}_2^{\epsilon}\oplus 1,I_2\oplus 0)$, which yields $(I_3,x\oplus y\oplus 0)$, $(-1\oplus I_2,x\oplus y\oplus 0)$, $(-1\oplus I_2,0\oplus x\oplus y)$, 
$\big(-1\oplus I_2,\tfrac{1}{2\overline{\xi}|\xi|}
\begin{bsmallmatrix}
\xi+\overline{\xi} & \overline{\xi}-\xi\\
\overline{\xi}-\xi & \xi+\overline{\xi}
\end{bsmallmatrix}\oplus 0\big)$, $(I_2\oplus 0,x\oplus 0\oplus 1)$, $(1\oplus  0_2,0\oplus I_2)$, $(-1\oplus 1\oplus 0,x\oplus 0\oplus y)$, 
$\big(-1\oplus I_2, \tfrac{1}{2x-i}
\begin{bsmallmatrix}
\frac{4x}{w-1} & -\frac{i}{x}\\
-\frac{i}{x} & \frac{4x}{w+1}
\end{bsmallmatrix}\oplus 0\big)$, $(I_2\oplus 0, 0\oplus 
\begin{bsmallmatrix}
0 & 1\\
1 & 0
\end{bsmallmatrix}$, $x,y>0$, $w=\sqrt{1+4x^2}$, $\xi \in \mathbb{C}\setminus \mathbb{R}$.

The pair $(b\oplus \mathcal{I}_2,1\oplus 0_2)$, $b\in \mathbb{R}$, transforms to $(1\oplus  0_2,0\oplus 1\oplus 0)$, $(I_2\oplus 0,0_2\oplus 1)$, $(I_2\oplus 0,x\oplus 0_2)$, $(-1\oplus 1\oplus 0,0_2\oplus 1)$, 
$(-1\oplus  I_2,0\oplus x\oplus 0)$, $(I_3,x\oplus 0_2)$, $(-1\oplus I_2,x\oplus 0_2)$, $x>0$.

Finally, from $(\mathcal{H}_3^{\epsilon},I_3)$ we get the following: $(I_3,x\oplus y\oplus z)$, $(-1\oplus I_2,x\oplus y\oplus z)$, 
$\big(-1\oplus 1\oplus 0,\tfrac{1}{2\overline{\xi}|\xi|}
\begin{bsmallmatrix}
\xi+\overline{\xi} & \overline{\xi}-\xi\\
\overline{\xi}-\xi & \xi+\overline{\xi}
\end{bsmallmatrix}\oplus 1\big)$, 
$\big(-1\oplus I_2,\tfrac{1}{2\overline{\xi}|\xi|}
\begin{bsmallmatrix}
\xi+\overline{\xi} & \overline{\xi}-\xi\\
\overline{\xi}-\xi & \xi+\overline{\xi}
\end{bsmallmatrix}\oplus x\big)$, 
$(1\oplus  0_2,x\oplus I_2)$, $(I_2\oplus 0,x\oplus y\oplus 1)$, $(-1\oplus 1\oplus 0,x\oplus y\oplus 1)$, 
$\big(1\oplus 0_2, 
\begin{bsmallmatrix}
0 & 1\\
1 & 0
\end{bsmallmatrix}\oplus 1\big)$, $(0_3,I_3)$,
$\big(-1\oplus 1\oplus 0,x\oplus 
\begin{bsmallmatrix}
0 & 1\\
1 & 0
\end{bsmallmatrix}\big)$, 
$\big(-1\oplus 1 \oplus 0, \tfrac{1}{2x-i}
\begin{bsmallmatrix}
\frac{4x}{w-1} & -\frac{i}{x}\\
-\frac{i}{x} & \frac{4x}{w+1}
\end{bsmallmatrix}\oplus 1\big)$,
$\big(-1\oplus I_2, \tfrac{1}{2x-i}
\begin{bsmallmatrix}
\frac{4x}{w-1} & -\frac{i}{x}\\
-\frac{i}{x} & \frac{4x}{w+1}
\end{bsmallmatrix}\oplus x\big)$, 
%$\big(-1\oplus I_2, \tfrac{1}{2x-i}\left[
%\begin{array}{cc}
%\frac{4x}{w+1} & -\frac{i}{x}\\
%-\frac{i}{x} & \frac{4x}{w-1}
%\end{array}
%\right]\oplus x\big)$, 
$\frac{1}{2}
\begin{bsmallmatrix}
i & -i & \sqrt{2}\\
-i & i & \sqrt{2}\\
\sqrt{2} & \sqrt{2} & 0
\end{bsmallmatrix}$, 
$x,y,z>0$, $w=\sqrt{1+4x^2}$.
\end{proof}

The proposition can be relatively easily generalized to the case of $4\times 4$ matrices. However, at the time of this writing we are not yet able to give a result for $n\times n$ matrices with $n\geq 5$ in a sufficiently nice form.

\begin{remark}
To some extend, Theorem \ref{clasB} could also be applied to tell us something about a more general situation for a pair of matrices $(A,B)$ with $A$ arbitrary. Indeed, $A$ can be written in a unique way as $A=H_1+iH_2$ with $H_1,H_2$ Hermitian, and then one of the matrices $H_1$ or $H_2$ could be put into the normal form, while keeping $B$ intact. It reduces the number of parameters by roughly one half. We also note here that the method, which was used in \cite{Coff} for the case of $2\times 2$ matrices (i.e. putting $A$ into a nice form first, and then $T$-conjugating $B$ by the matrices preserving $A$ under $*$-conjugation), does not seem to adapt to the case of $3\times 3$ matrices due to the involved computations.  
\end{remark}

\section{Appendix}\label{appendix}

\subsection{Consimilarity}

\begin{trditev}\label{ConSP} 
\begin{enumerate}
\item (\cite[Theorem 4.1]{HongHorn}) Let $A$, $B$ be two $n\times n$ matrices. Then $A\overline{A}$, $B\overline{B}$ are similar precisely when $A$, $B$ are consimilar and satisfy the alternating product rank condition ($\rank \Pi^k(A\overline{A})=\rank \Pi^k (B\overline{B})$).
\item (\cite[Corollary 4.4]{HongHorn}) Let $A$, $B$ be two $n\times n$ matrices and let $C$ be a $m\times m $ matrix. Then $A\oplus C$ is consimilar to $B\oplus C$ if and only if $A$ is consimilar to $B$.
\end{enumerate}
\end{trditev}

\subsection{Matrices of double size}\label{MDS}

\begin{lemma}\label{ConForm}
Assume $A=A_1+iA_2$, where $A_1,A_2$ are real $n\times n$ matrices and let  
$\widehat{A}=
\begin{bsmallmatrix}
0 & \overline{A} \\
A & 0
\end{bsmallmatrix}$
and 
$\widetilde{A}=
\begin{bsmallmatrix}
A_1 & A_2 \\
A_2 & -A_1
\end{bsmallmatrix}$ be the corresponding two block-matrices. Then
\begin{enumerate}
\item \cite[Problem 1.3.P21]{HornJohn} $\widetilde{A}$ and $\widehat{A}$ are similar and $\Delta_{\widetilde{A}}(\lambda)=\Delta_{\widehat{A}}(\lambda)=\Delta_{A\overline{A}}(\lambda^2)$.\\
(In particular, if $\lambda$ is the eigenvalue of $A\overline{A}$, $\pm\lambda{}$ is the eigenvalue of $\widehat{A}$ or $\widetilde{A}$; the non-real eigenvalues of $A\overline{A}$ occur in conjugate pairs.)
\item \label{CF2}\cite[Theorems 5 and 6]{Bern} The real Jordan form of $\widetilde{A}$ or $\widehat{A}$ is equal to $\mathcal{J}(\widetilde{A})=\mathcal{J}(\widehat{A})=\mathcal{J}^+\oplus \mathcal{J}^-$, where $\mathcal{J}^+$ ($\mathcal{J}^-$) consists of real Jordan blocks $J_{\alpha_j}(\lambda_,1)$ or $J_{\beta_k}(\Lambda_k,I_2)$ corresponding to the eigenvalues with non-negative real part ($J_{\alpha_j}(-\lambda_,1)$ or $J_{\beta_k}(-\Lambda_k,-I_2)$ with non-positive real part), and in addition each with only half of the blocks for eigenvalues with no real part (they occur in even pairs). 
Moreover, $A\overline{A}$ is similar to $(\mathcal{J}^+)^2$.
%$A$ is consimilar to $\mathcal{J}^+$.
\end{enumerate}
\end{lemma}

\begin{proof}[Sketch of the proof of Lemma \ref{ConForm}(\ref{CF2})]
Let $J$ be the non-singular part of the Jordan canonical form of the matrix $A\overline{A}$, and let the matrix $S$ consists of its corresponding generalized eigenvectors (i.e. $A\overline{A}S=SJ$). If $J^{\frac{1}{2}}$ is the square root of $J$ and such that the real parts of its eigenvalues are non-negative. Then
\[
\begin{bmatrix}
0 & \overline{A} \\
A & 0
\end{bmatrix}
\begin{bmatrix}
\overline{A}SJ^{-\frac{1}{2}} & -\overline{A}SJ^{-\frac{1}{2}} \\
S & S
\end{bmatrix}
=
\begin{bmatrix}
\overline{A}SJ^{-\frac{1}{2}} & -\overline{A}SJ^{-\frac{1}{2}} \\
S & S
\end{bmatrix}
\begin{bmatrix}
J^{\frac{1}{2}} & 0 \\
0 & -J^{\frac{1}{2}}
\end{bmatrix},
\]

On the other hand, it is not difficult to see that for Jordan block $J_m(\lambda,1)$ of $\widetilde{A}$ corresponding to the real eigenvalue $\lambda$ we have $\widetilde{A}T=TJ(\lambda)$ precisely when $AW=\overline{W}J(\lambda)$, where $T$ consist of generalized eigenvectors $t_j=
\begin{bsmallmatrix}
u_j \\
v_j
\end{bsmallmatrix}$, $j\in\{1,\ldots,m\}$, of $\widetilde{A}$ with respect to $\lambda$ ($(\widetilde{A}-\lambda I)t_{j+1}=t_j$), and $w_j=u_j-iv_j$, for all $j$, are the columns of $W$. Observe that $t_j'=
\begin{bsmallmatrix}
(-1)^{j}v_j \\
(-1)^{j+1}u_j
\end{bsmallmatrix}$, $j\in\{1,\ldots,m\}$, forming the matrix $T'$, are the generalized eigenvectors of $\widetilde{A}$ corresponding to the eigenvalue $-\lambda$, thus $\widetilde{A}T'=T'J(-\lambda)$ and $AW'=\overline{W}'J(-\lambda)$ with  $w_j'=(-1)^{j}v_j+i(-1)^{j+1}u_j$ the columns of $W'$. In particular, the Jordan blocks of $\widetilde{A}$ corresponding to the zero-eigenvalues occur in pairs. Also, $A\overline{A}\overline{W}=\overline{W}(J(0))^2$, $A\overline{A}\overline{W}'=\overline{W}'(J(0))^2$ and this concludes the proof of similarity for $A\overline{A}$ and $(\mathcal{J}^+)^2$.
\end{proof}

\begin{lemma}\label{lemaL4}
Let $A$ be an $n\times n$ non-singular matrix. Then $AA^*$ is positive definite, and assume that $u_1,\ldots,u_n\in \mathbb{C}^n$ are its eigenvectors, corresponding to eigenvalues $\lambda_1,\ldots,\lambda_n$. Furthermore, for  $j\in \{1,\ldots,m\}$, the pairs $
\begin{bsmallmatrix}
u_j \\
\frac{1}{\sqrt{\lambda_j}}A^{*}u_j
\end{bsmallmatrix}$, $
\begin{bsmallmatrix}
u_j \\
-\frac{1}{\sqrt{\lambda_j}}A^{*}u_j
\end{bsmallmatrix}
\in \mathbb{C}^{2n}$, are the eigenvectors 
with the corresponding eigenvalues $\pm\sqrt{\lambda_j}$ of the matrix $
\begin{bsmallmatrix}
0 & A \\
A^{*} & 0
\end{bsmallmatrix}$.
In particular, if $A=iB$ for some hermitian matrix $B$, then $AA^*=B^2$ and the pairs $
\begin{bsmallmatrix}
u_j \\
\frac{1}{\sqrt{\lambda_j}}iB u_j
\end{bsmallmatrix}$, 
$
\begin{bsmallmatrix}
u_j \\
-\frac{1}{\sqrt{\lambda_j}}iB u_j
\end{bsmallmatrix}
\in \mathbb{C}^{2n}$, are the eigenvectors 
with the corresponding eigenvalues $\pm\sqrt{\lambda_j}$ of the matrix $\begin{bsmallmatrix}
0 & iB \\
-iB & 0
\end{bsmallmatrix}$.
\end{lemma}

\end{document}